\documentclass[12pt,oneside]{amsart}
\usepackage{amssymb, amsmath, amsthm}

\usepackage{epsfig}
\usepackage{epstopdf}
\usepackage{graphicx}
\usepackage{color}
\usepackage{mathptmx}

\usepackage{pinlabel}
\theoremstyle{plain}
\newtheorem{theorem}{Theorem}[section]

\newtheorem*{MainThm}{Main Theorem}
\newtheorem*{main result}{Simplified Corollary \ref{main corollary 1}}
\newtheorem*{mainthmsimp}{Simplified Theorem \ref{Thm: MSC}}
\newtheorem*{TangleBound}{Theorem \ref{Thm: Meridional Planar}}
\newtheorem*{theorem*}{Theorem}
\newtheorem*{conjecture}{Conjecture}
\newtheorem*{TangleCalcI}{Tangle Calculations I}
\newtheorem*{TangleCalcII}{Tangle Calculations II}

\newtheorem{proposition}[theorem]{Proposition}
\newtheorem{corollary}[theorem]{Corollary}
\newtheorem{lemma}[theorem]{Lemma}

\theoremstyle{definition}

\newtheorem*{definition}{Definition}
\newtheorem*{remark}{Remark}

\newtheorem*{example}{Example}

\newcommand{\Z}{\mathbb Z}
\newcommand{\N}{\mathbb N}

\newcommand{\nil}{\varnothing}
\newcommand{\tild}{\widetilde}
\newcommand{\wihat}{\widehat}
\newcommand{\defn}[1]{\textbf{#1}}
\newcommand{\boundary}{\partial}

\newcommand{\mc}[1]{\mathcal{#1}}
\newcommand{\ob}[1]{\overline{#1}}
\newcommand{\cls}{\operatorname{cl}} %closure
\newcommand{\genus}{\operatorname{genus}} %genus
\newcommand{\interior}{\operatorname{int}} %interior
\newcommand{\inter}[1]{\mathring{#1}}

\def\hyph{-\penalty0\hskip0pt\relax}

      \makeatletter
      \def\@setcopyright{}
      \def\serieslogo@{}
      \makeatother
      
\begin{document}
   % title

   \title[Boring Split Links]{Boring Split Links}
   \author{Scott A Taylor}
   \email{sataylor@colby.edu}
   \thanks{This research was partially supported by a grant from the National Science Foundation.}
   % Note that the short title for running heads goes in square
   % brackets.  This is optional.  The long title goes in curly
   % braces.  In the long title, line breaks are indicated by \\

   % today's date, or fill in whatever date you prefer
   \date{\today}
\begin{abstract}
Boring is an operation which converts a knot or two\hyph component link in a 3--manifold into another knot or two-component link. It generalizes rational tangle replacement and can be described as a type of 2--handle attachment. Sutured manifold theory is used to study the existence of essential spheres and planar surfaces in the exteriors of knots and links obtained by boring a split link. It is shown, for example, that if the boring operation is complicated enough, a split link or unknot cannot be obtained by boring a split link. Particular attention is paid to rational tangle replacement. If a knot is obtained by rational tangle replacement on a split link, and a few minor conditions are satisfied, the number of boundary components of a meridional planar surface is bounded below by a number depending on the distance of the rational tangle replacement. This result is used to give new proofs of two results of Eudave-Mu\~noz and Scharlemann's band sum theorem.
\end{abstract}

 \maketitle
 
\section{Introduction}

\subsection{Refilling and Boring}
Given a genus two handlebody $W$ embedded in a 3--manifold $M$, a knot or two-component link can be created by choosing an essential disc $\alpha \subset W$ and boundary-reducing $W$ along $\alpha$. Then $W-\inter{\eta}(\alpha)$ is the regular neighborhood of a knot or link $L_\alpha$.  We say that the exterior $M[\alpha]$ of this regular neighborhood is obtained by \defn{refilling} the meridian disc $\alpha$.  Similarly, given a knot or link $L_\alpha \subset M$ we can obtain another knot or link $L_\beta$ by the following process:
\begin{enumerate}
\item Attach an arc to $L_\alpha$ forming a graph
\item Thicken the graph to form a genus two handlebody $W$
\item Choose a meridian $\beta$ for $W$ and refill $\beta$.
\end{enumerate}

Refilling the meridian $\alpha$ of the attached arc returns $L_\alpha$.  Any two knots in $S^3$ can be related by such a move if we allow $\alpha$ and $\beta$ to be disjoint: just let $W$ be a neighborhood of the wedge of the two knots.  We'll restrict attention, therefore, to meridians of $W$ which cannot be isotoped to be disjoint.  If a knot or link $L_\beta$ can be obtained from $L_\alpha$ by this operation say that $L_\beta$ is obtained by \defn{boring} $L_\alpha$.  Since the relation is symmetric we may also say that $L_\alpha$ and $L_\beta$ are \defn{related by boring}.

Boring generalizes several well-known operations in knot theory. Band sums, crossing changes, generalized crossing changes, and, more generally, rational tangle replacement can all be realized as boring.  The band move from the Kirby calculus \cite{FR,K} is also a type of boring.  If $W$ is the standard genus 2 handlebody in $S^3$ and $L_\alpha$ is the unlink of two components then all tunnel number 1 knots can be obtained by boring $L_\alpha$ using $W$.

If $L_\alpha$ and $L_\beta$ are related by boring, it is natural to ask under what circumstances both links can be split, both the unknot, both composite, etc.  Many of these questions have been effectively addressed for special types of boring, such as rational tangle replacement \cite{EM2}.  Following \cite{S5}, this paper will focus on the exteriors $M[\alpha]$ and $M[\beta]$ of the knots $L_\alpha$ and $L_\beta$, respectively.  In that paper, Scharlemann conjectured that, with certain restrictions (discussed in Section \ref{RM}), if $M[\alpha]$ and $M[\beta]$ are both reducible or boundary-reducible then either $W$ is an unknotted handlebody in $S^3$ or $\alpha$ and $\beta$ are positioned in a particularly nice way in $W$.  He was able to prove his conjecture (with slightly varying hypotheses and conclusions) when $M - \inter{W}$ is boundary-reducible, when $|\alpha \cap \beta| \leq 2$, or when one of the discs is separating.  

This paper looks again at these questions and completes, under stronger hypotheses, the proof of Scharlemann's conjecture except when $M = S^3$ and $M[\alpha]$ and $M[\beta]$ are solid tori.  With these stronger hypotheses, however, we reach conclusions which are stronger than those obtained in \cite{S5}.  Even in the one situation which is not completed, we do gain significant insight. The remaining case is finally completed in \cite{T}. Here is a simplified version of one of the main theorems:

\begin{mainthmsimp}
Suppose that $M$ is $S^3$ or the exterior of a link in $S^3$ and that $M - \inter{W}$ is irreducible and boundary-irreducible. If $\alpha$ and $\beta$ cannot be isotoped to be disjoint, then at least one of $M[\alpha]$ or $M[\beta]$ is irreducible. Furthermore, if one is boundary-reducible (e.g. a solid torus) then the other is not reducible.
\end{mainthmsimp}

The conclusions of Theorem \ref{Thm: MSC} are an ``arc-version'' of the conclusions of the main theorem of \cite{S4} which considers surgeries on knots producing reducible 3--manifolds.  The methods of this paper are similar in outline to those of \cite{S4} but differ in detail. 

Perhaps the most interesting application of these techniques to rational tangle replacement is the following theorem, which generalizes some results of Eudave-Mu\~noz and Scharlemann:
\begin{TangleBound}
Suppose that $L_\beta$ is a knot or link in $S^3$ and that $B' \subset S^3$ is a 3--ball intersecting $L_\beta$ so that $(B',B' \cap L_\beta)$ is a rational tangle. Let $(B',r_\alpha)$ be another rational tangle of distance $d \geq 1$ from $r_\beta = B' \cap L_\beta$ and let $L_\alpha$ be the knot obtained by replacing $r_\beta$ with $r_\alpha$. Let $(B,\tau) = (S^3 - \inter{B}',L_\beta - \inter{B}')$. Suppose that $L_\alpha$ is a split link and that $(B,\tau)$ is prime. Then $L_\beta$ is not a split link or unknot. Furthermore, if $L_\beta$ has an essential properly embedded meridional planar surface with $m$ boundary components, it contains such a surface $\ob{Q}$ with $|\boundary \ob{Q}| \leq m$ such that either $\ob{Q} \subset B$ or
\[
|\ob{Q} \cap \boundary B|(d - 1) \leq |\boundary \ob{Q}| - 2
\]
\end{TangleBound}

One consequence of this is a new proof of Scharlemann's band sum theorem: If the unknot is obtained by attaching a band to a split link then the band sum is the connected sum of unknots. This, and other rational tangle replacement theorems, are proved in Section \ref{RTR}.

The main tool in this paper is Scharlemann's combinatorial version of Gabai's sutured manifold theory.  The relationship of this paper to \cite{S5}, where Scharlemann states his conjecture about refilling meridians, is similar to the relationship between Gabai's and Scharlemann's proofs of the band sum theorem.  In \cite{S1}, Scharlemann proved that the band sum of two knots is unknotted only if it is the connect sum of two unknots.  Later Scharlemann and Gabai simultaneously and independently proved that 
\[\genus(K_1 \#_b K_2) \geq \genus(K_1) + \genus(K_2),\]
where $\#_b$ denotes a band sum.  Gabai \cite{G1} used sutured manifold theory to give a particularly simple proof. Scharlemann's proof \cite{S3} uses a completely combinatorial version of sutured manifold theory. Since rational tangle replacement is a special type of boring, a similar relationship also holds between this paper and some of Eudave-Mu\~noz's \cite{EM2} extensions of the original band sum theorem.  The techniques of this paper can be specialized to rational tangle replacement to recapture and extend some, but not all, of his results.  In \cite{S5}, Scharlemann suggests that sutured manifold theory might contribute to a solution to his conjecture.  This paper vindicates that idea.

The paper \cite{T} uses sutured manifold theory in a different way. The approaches of the that paper and this paper are often useful in different circumstances. For example, the approach used in this paper is more effective for studying the existence of certain reducing spheres in a manifold obtained by refilling meridians and for studying non-separating surfaces which are not homologous to a surface with interior disjoint from $W$. The approach of \cite{T} is more effective for the study of essential discs and separating surfaces.

\subsection{Notation}
We work in the PL or smooth categories.  All manifolds and surfaces will be compact and orientable, except where indicated.  $|A|$ denotes the number of components of $A$. If $A$ and $B$ are embedded curves on a surface, $|A \cap B|$ will generally be assumed to be minimal among all curves isotopic to $A$ and $B$.  For a subcomplex $B \subset A$, $\eta(B)$ denotes a closed regular neighborhood of $B$ in $A$.  $\inter{B}$ and $\interior B$ both denote the interior of $B$ and $\cls(B)$ denotes the closure of $B$. $\boundary B$ denotes the boundary of $B$. All homology groups have $\Z$ (integer) coefficients.

\subsection{Acknowledgements}
The work in this paper is part of my Ph.D. dissertation at the University of California, Santa Barbara.  I am grateful to my advisor, Martin Scharlemann, for suggesting I learn sutured manifold theory and for many helpful conversations. This paper benefited greatly from his comments on early drafts. I am also grateful to Ryan Blair and Robin Wilson for our conversations and to Stephanie Taylor for suggesting the ``boring'' terminology. Portions of this paper were written at Westmont College and the research was partially supported by a grant from the National Science Foundation.

%%%%%%%%%%%%%%%%%%%% SUTURED MANIFOLD THEORY %%%%%%%%%%%%%%%%%%%%%%%%

\section{Sutured Manifold Theory}
We begin by reviewing a few relevant concepts from combinatorial sutured manifold theory \cite{S3}.

\subsection{Definitions}
A \defn{sutured manifold} is a triple $(N,\gamma,\psi)$ where $N$ is a compact, orientable 3--manifold, $\gamma$ is a collection of oriented simple closed curves on $\boundary N$, and $\psi$ is a properly embedded 1--complex. $T(\gamma)$ denotes a collection of torus components of $\boundary N$. The curves $\gamma$ divide $\boundary N - T(\gamma)$ into two surfaces which intersect along $\gamma$. Removing $\inter{\eta}(\gamma)$ from these surfaces creates the surfaces $R_+(\gamma)$ and $R_-(\gamma)$. Let $A(\gamma) = \eta(\gamma)$.

For an orientable, connected surface $S \subset N$ in general position with respect to $\psi$, we define
\[
\chi_\psi(S) = \max\{0,|S \cap \psi| - \chi(S)\}.
\]
If $S$ is disconnected, $\chi_\psi(S)$ is the sum of $\chi_\psi(S_i)$ for each component $S_i$. For a class $[S] \in H_2(N,X)$, $\chi_\psi([S])$ is defined to be the minimum of $\chi_\psi(S)$ over all embedded surfaces $S$ representing $[S]$. If $\psi = \nil$, then $\chi_\psi(\cdot)$ is the Thurston norm.

Of utmost importance is the notion of $\psi$--tautness for both surfaces in a sutured manifold $(N,\gamma,\psi)$ and for a sutured manifold itself. Let $S$ be a properly embedded surface in $N$.
\begin{itemize}
\item $S$ is $\psi$--minimizing in $H_2(N,\boundary S)$ if $\chi_\psi(S) = \chi_\psi[S,\boundary S]$.
\item $S$ is $\psi$--incompressible if $S - \psi$ is incompressible in $N - \psi$.
\item $S$ is $\psi$--taut if it is $\psi$--incompressible, $\psi$--minimizing in $H_2(N,\boundary S)$ and each edge of $\psi$ intersects $S$ with the same sign. If $\psi = \nil$ then we say either that $S$ is $\nil$--taut or that $S$ is taut in the Thurston norm.
\end{itemize}

A sutured manifold $(N,\gamma,\psi)$ is $\psi$--taut if
\begin{itemize}
\item $\boundary \psi$ (i.e. valence one vertices) is disjoint from $A(\gamma) \cup T(\gamma)$
\item $T(\gamma)$, $R_+(\gamma)$, and $R_-(\gamma)$ are all $\psi$--taut.
\item $N - \psi$ is irreducible.
\end{itemize}

The final notion that is important for this paper is the concept of a conditioned surface. A \defn{conditioned surface} $S \subset N$ is an oriented properly embedded surface such that:
  \begin{itemize}
  \item If $T$ is a component of $T(\gamma)$ then $\boundary S \cap T$ consists of coherently oriented parallel circles.
  \item If $A$ is a component of $A(\gamma)$ then $S \cap A$ consists of either circles parallel to $\gamma$ and oriented the same direction as $\gamma$ or arcs all oriented in the same direction.
  \item No collection of simple closed curves of $\boundary S \cap R(\gamma)$ is trivial in $H_1(R(\gamma),\boundary R(\gamma))$.
  \item Each edge of $\psi$ which intersects $S \cup R(\gamma)$ does so always with the same sign.
  \end{itemize}

Conditioned surfaces, along with product discs and annuli, are the surfaces along which a taut sutured manifold is decomposed to form a taut sutured manifold hierarchy. A hierarchy can be taken to be ``adapted'' to a parameterizing surface, that is, a surface $Q \subset N - \inter{\eta}(\psi)$ no component of which is a disc disjoint from $\gamma \cup \eta(\psi)$. The ``index'' of a parameterizing surface is a certain number associated to $Q$ which does not decrease as $Q$ is modified during the hierarchy.

\subsection{Satellite knots have Property P}
It will be helpful to review the essentials of the proof of \cite[Theorem 9.1]{S3}, where it is shown that satellite knots have property P.

In that theorem, a 3--manifold $N$ with $\boundary N$ a torus is considered.  It is assumed that $H_1(N)$ is torsion-free and that $k \subset N$ is a knot in $N$ such that $(N,\nil)$ is a $k$--taut sutured manifold.  Suppose that some non-trivial surgery on $k$ creates a manifold which has a boundary\hyph reducing disc $\ob{Q}$ and still has torsion-free first homology. The main goal is to show that $(N,\nil)$ is $\nil$--taut. The surface $Q = \ob{Q} - \inter{\eta}(k)$ acts as a parameterizing surface for a $k$--taut sutured manifold hierarchy of $N$.  At the end of the hierarchy, there is at least one component containing pieces of $k$.  A combinatorial argument using the assumption that $H_1(N)$ is torsion-free shows that, in fact, the last stage of the hierarchy is $\nil$--taut.  Sutured manifold theory then shows that the original manifold $N$ is $\nil$--taut, as desired.  This argument is extended in \cite{S4} to study surgeries on knots in 3--manifolds which produce reducible 3--manifolds.  In that paper, the surface $\ob{Q}$ can be either a $\boundary$-reducing disc or a reducing sphere.

This paper extends these techniques in two other directions.  First, we use an arc $\ob{\alpha} \subset M[\alpha]$ in place of the knot $k \subset N$.  Second, we develop criteria which allow the surface $\ob{Q} \subset M[\beta]$ to be any of a variety of surfaces including essential spheres and discs. Section \ref{Constructing Q} shows how to construct a useful surface $\ob{Q}$. Section \ref{Sutures} discusses the placement of sutures on $\boundary M[\alpha]$. This allows theorems about sutured manifolds to be phrased without reference to sutured manifold terminology. Section \ref{RM} applies the sutured manifold results in order to (partially) answer Scharlemann's conjecture about refilling meridians of genus two handlebodies. Section \ref{RTR} uses the technology to reprove three classical theorems about rational tangle replacement and prove a new theorem about essential meridional surfaces in the exterior of a knot or link obtained by boring a split link. Finally, Section \ref{Intersections} shows how the sutured manifold theory results of this paper can significantly simplify certain combinatorial arguments.

\section{Attaching a 2--handle}\label{2--handle}

Let $N$ be a compact orientable 3--manifold containing a component $F \subset \boundary N$ of genus at least two. Let $a \subset F$ be an essential closed curve and let $\mc{B} = \{b_1, \hdots, b_{|\mc{B}|}\}$ be a collection of disjoint, pairwise non-parallel essential closed curves in $F$ isotoped so as to intersect $a$ minimally. Suppose that $\gamma \subset \boundary N$ is a collection of simple closed curves, disjoint from $a$, such that $(N,\gamma \cup a)$ is a taut sutured manifold and $\gamma$ intersects the curves of $\mc{B}$ minimally. Let $\Delta_i = |b_i \cap a|$ and $\nu_i = |b_i \cap \gamma|$. 

Suppose that $Q \subset N$ is a surface with $q_i$ boundary components parallel to the curve $b_i$, for each $1 \leq i \leq |\mc{B}|$. Let $\boundary_0 Q$ be the components of $\boundary Q$ which are not parallel to any $b_i$. Assume that $\boundary Q$ intersects $\gamma \cup a$ minimally. Define $\Delta_\boundary = |\boundary_0 Q \cap a|$ and $\nu_\boundary = |\boundary_0 Q \cap \gamma|$. We need two definitions. The first defines a specific type of boundary compression and the second (as we shall see) is related to the notion of ``Scharlemann cycle''.

\begin{definition}
An \defn{$a$--boundary compressing disc} for $Q$ is a boundary compressing disc $D$ for which $\boundary D \cap F$ is a subarc of some essential circle in $\eta(a)$.
\end{definition}

\begin{definition}
An \defn{$a$--torsion $2g$--gon} is a disc $D \subset N$ with $\boundary D \subset F \cup Q$ consisting of $2g$ arcs labelled around $\boundary D$ as $\delta_1, \epsilon_1, \hdots, \delta_g, \epsilon_g$. The labels are chosen so that $\boundary D \cap Q = \cup \delta_i$ and $\boundary D \cap F = \cup \epsilon_i$. We require that each $\epsilon_i$ arc is a subarc of some essential simple closed curve in $\eta(a)$ and that the $\epsilon_i$ arcs are all mutually parallel as oriented arcs in $F - \boundary Q$. Furthermore we require that attaching to $Q$ a rectangle in $F - \boundary Q$ containing all the $\epsilon_i$ arcs produces an orientable surface.
\end{definition}

\begin{example}
Figure \ref{Fig: aTorsion2nGon} shows a hypothetical example. The surface outlined with dashed lines is $Q$. It has boundary components on $F$. There are two such boundary components pictured. The curve running through $Q$ and $F$ could be the boundary of an $a$--torsion $4$--gon. Notice that the arcs $\epsilon_1$ and $\epsilon_2$ are parallel and oriented in the same direction. Attaching the rectangle containing those arcs as two of its edges to $Q$ produces an orientable surface.
\end{example}

\begin{figure}[ht]
\labellist
\hair 2pt
\pinlabel $\delta_1$ [b] at 292 365
\pinlabel $Q$ at 540 321
\pinlabel $\delta_2$ [b] at 165.5 275
\pinlabel $\epsilon_2$ [r] at 251 216
\pinlabel $\epsilon_1$ [l] at 271 216
\pinlabel $F$ at 217 62
\endlabellist
\center
\includegraphics[scale=0.6]{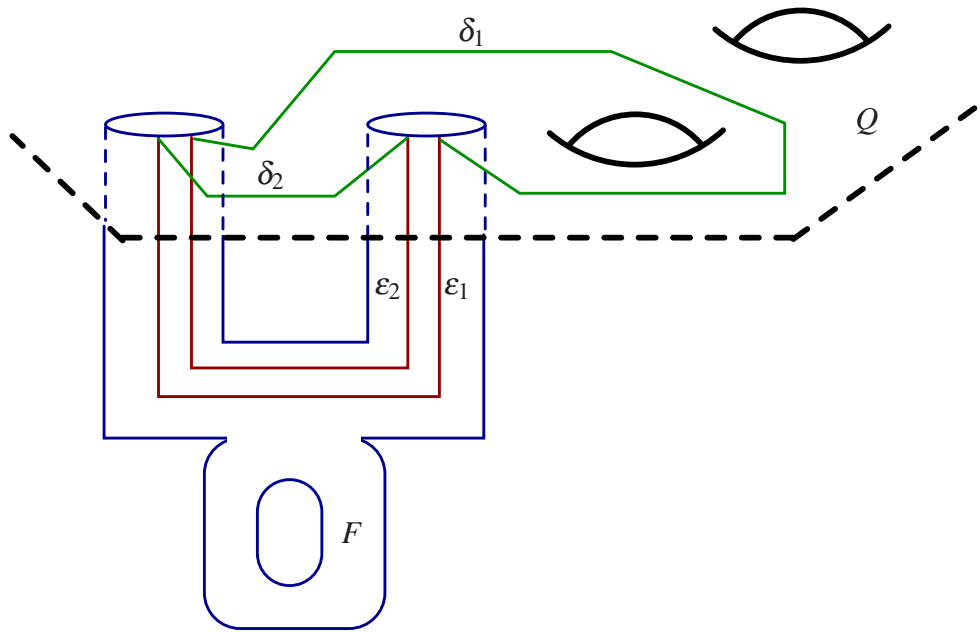}
%\scalebox{0.6}{\input{aTorsion2nGon.eps_t}}
\caption{The boundary of an $a$--torsion $4$--gon.}
\label{Fig: aTorsion2nGon}
\end{figure}

\begin{remark}
Notice that an $a$--torsion $2$--gon is an $a$--boundary compressing disc.
\end{remark}

If 2--handles are attached to each curve $b_i$ for $1 \leq i \leq |\mc{B}|$, a 3--manifold $N[\mc{B}]$ is obtained. Each component of $\boundary Q - \boundary_0 Q$ bounds a disc in $N[\mc{B}]$. Let $\ob{Q}$ be the result of attaching a disc to each component of $\boundary Q - \boundary_0 Q$. Then $\boundary \ob{Q} = \boundary_0 Q$. We will usually also attach 3--balls to spherical components of $\boundary N[\mc{B}]$.

\begin{remark}
The term ``$a$--torsion $2g$--gon'' is chosen because in certain cases (but not all) the presence of an $a$--torsion $2g$--gon with $g \geq 2$ guarantees that $N[\mc{B}]$ has torsion in its first homology.
\end{remark}

Define
\[
K(\ob{Q}) = \sum_{i=1}^{|\mc{B}|}q_i(\Delta_i - 2) + \Delta_\boundary - \nu_\boundary
\]

We will use the surface $Q$ to study the effects of attaching a 2--handle $\alpha \times I$ to a regular neighborhood of the curve $a \subset F$. Let $N[a]$ denote the resulting 3--manifold. Perform the attachment so that the 2--disc $\alpha$ has boundary $a$. Let $\ob{\alpha}$ denote the arc which is the cocore of the 2--handle $\alpha \times I$.

We can now state our main sutured manifold theory result. It is an adaptation of Theorem 9.1 of \cite{S3}.

\begin{MainThm}[cf. {\cite[Theorem 9.1]{S3} and \cite[Proposition 4.1]{S4}}] Suppose that $(N[a],\gamma)$ is $\ob{\alpha}$--taut, that $Q$ is incompressible, and that $Q$ contains no disc or sphere component disjoint from $\gamma \cup a$. Furthermore, suppose that one of the following holds:
\begin{itemize}
\item $N[a]$ is not $\nil$--taut
\item there is a conditioned  $\ob{\alpha}$--taut surface $S \subset N[a]$ which is not $\nil$--taut.
\item $N[a]$ is homeomorphic to a solid torus $S^1 \times D^2$ and $\ob{\alpha}$ cannot be isotoped so that its projection to the $S^1$ factor is monotonic.
\end{itemize}
Then at least one of the following holds:
\begin{itemize}
\item There is an $a$--torsion $2g$--gon for $Q$ for some $g \in \N$
\item $H_1(N[a])$ contains non-trivial torsion
\item $-2\chi(\ob{Q}) \geq K(\ob{Q})$.
\end{itemize}
\end{MainThm}

\begin{remark}
If $\ob{\alpha}$ can be isotoped to be monotonic in the solid torus $N[a]$ then it is, informally, a ``braided arc''.  The contrapositive of this aspect of the theorem is similar to the conclusion in \cite{G2} and \cite{S4} that if a non-trivial surgery on a knot with non-zero wrapping number in a solid torus produces a solid torus then the knot is a 0 or 1--bridge braid.
\end{remark}

The remainder of this section proves the theorem. Following \cite{S4}, define a \defn{Gabai disc} for $Q$ to be an embedded disc $D \subset N[a]$ such that
\begin{itemize}
\item $|\ob{\alpha} \cap \inter{D}| > 0$ and all points of intersection have the same sign of intersection
\item $|Q \cap \boundary D| < |\boundary Q \cap \eta(a)|$
\end{itemize}

The next proposition points out that the existence of a Gabai disc guarantees the existence of an $a$--boundary compressing disc or an $a$--torsion $2g$--gon.

\begin{proposition}
If there is a Gabai disc for $Q$ then there is an $a$--torsion $2g$--gon.
\end{proposition}
\begin{proof}
Let $D$ be a Gabai disc for $Q$. The intersection of $Q$ with $D$ produces a graph $\Lambda$ on $D$.  The vertices of $\Lambda$ are $\boundary D$ and the points $\ob{\alpha} \cap D$. The latter are called the \defn{interior vertices} of $\Lambda$.  The edges of $\Lambda$ are the arcs $Q \cap D$. A \defn{loop} is an edge in $\Lambda$ with initial and terminal points at the same vertex. A loop is \defn{trivial} if it bounds a disc in $D$ with interior disjoint from $\Lambda$.

To show that there is an $a$--torsion $2g$--gon for $Q$, we will show that the graph $\Lambda$ contains a ``Scharlemann cycle'' of length $g$. The interior of the Scharlemann cycle will be the $a$--torsion $2g$--gon. In our situation, Scharlemann cycles will arise from a labelling of $\Lambda$ which is slightly non-standard.  Traditionally, when $\ob{\alpha}$ is a knot instead of an arc, the labels on the endpoints of edges in $\Lambda$, which are used to define ``Scharlemann cycles'', are exactly the components of $\boundary Q$.  In our case, since each component of $\boundary Q$ likely intersects $a$ more than once we need to use a slightly different labelling.  After defining the labelling and the revised notion of ``Scharlemann cycle'', it will be clear to those familiar with the traditional situation that the new Scharlemann cycles give rise to the same types of topological conclusions as in the traditional setting.  The discussion is modelled on Section 2.6 of \cite{CGLS}. 

A Scharlemann cycle of length 1 is defined to be a trivial loop at an interior vertex of $\Lambda$ bounding a disc with interior disjoint from $\Lambda$. We now work toward a definition of Scharlemann cycles of length $g > 1$.  Without loss of generality, we may assume that $|\ob{\alpha} \cap D|\geq 2$. Recall that the arc $\ob{\alpha}$ always intersects the disc $D$ with the same sign.  There is, in $F$, a regular neighborhood $A$ of $a$ such that $D \cap F \subset A$.  We may choose $A$ so that $\boundary A \subset D \cap F$. Let $\boundary_\pm A$ be the two boundary components of $A$. The boundary components of $Q$ all have orientations arising from the orientation of $\ob{Q}$ and $\ob{\beta}$. We may assume by an isotopy that all the arcs $\boundary Q \cap A$ are fibers in the product structure on $A$. Cyclically around $A$ label the arcs of $\boundary Q \cap A$ with labels $c_1, \hdots, c_\mu$.  Let $\mc{C}$ be the set of labels.  Being a submanifold of $\boundary Q$, each arc is oriented.  Say that two arcs are \defn{parallel} if they run through $A$ in the same direction (that is, both from $\boundary_- A$ to $\boundary_+ A$ or both from $\boundary_+ A$ to $\boundary_- A$).  Call two arcs \defn{antiparallel} if they run through $A$ in opposite directions.  Note that since the orientations of $\inter{D} \cap F$ in $A$ are all the same, an arc intersects each component of $\inter{D} \cap F$ with the same algebraic sign.  

Call an edge of $\Lambda$ with at least one endpoint on $\boundary D$ a \defn{boundary edge} and call all other edges \defn{interior edges}.  As each edge of $\Lambda$ is an arc and as all vertices of $\Lambda$ are parallel oriented curves on $\boundary W$, an edge of $\Lambda$ must have endpoints on arcs of $\mc{C} = \{c_1, \hdots, c_\mu\}$ which are antiparallel.  We call this the \defn{parity principle} (as in \cite{CGLS}).  Label each endpoint of an edge in $\Lambda$ with the arc in $\mc{C}$ on which the endpoint lies. 

We will occasionally orient an edge $e$ of $\Lambda$; in which case, let $\boundary_- e$ be the tail and $\boundary_+ e$ the head.  A \defn{cycle} in $\Lambda$ is a subgraph homeomorphic to a circle.  An $x$--cycle is a cycle which, when each edge $e$ in the cycle is given a consistent orientation, has $\boundary_- e$ labelled with $x \in \mc{C}$.  Let $\Lambda'$ be a subgraph of $\Lambda$ and let $x$ be a label in $\mc{C}$.  We say that $\Lambda'$ satisfies condition $P(x)$ if: 

$P(x)$: For each vertex $v$ of $\Lambda'$ there exists an edge of $\Lambda'$ incident to $v$ with label $x$ connecting $v$ to an interior vertex. 

\begin{lemma}[{\cite[Lemma 2.6.1]{CGLS}}] \label{P-condition}
Suppose that $\Lambda'$ satisfies $P(x)$.  Then each component of $\Lambda'$ contains an $x$--cycle.
\end{lemma}

\begin{proof}
The proof is the same as in \cite{CGLS}.
\end{proof}

A \defn{Scharlemann cycle} is an $x$--cycle $\sigma$ where the interior of the disc in $D$ bounded by $\sigma$ is disjoint from $\Lambda$. See Figure \ref{Fig: Schcycle}. Since each intersection point of $D \cap \ob{\alpha}$ has the same sign, the set of labels on a Scharlemann cycle contains $x$ and precisely one other label $y$, a component of $\mc{C}$ adjacent to $x$ in $A$. The arc $y$ and the arc $x$ are antiparallel by the parity principle. The \defn{length} of the Scharlemann cycle is the number of edges in the $x$--cycle.

\begin{figure}[ht]
\labellist
\hair 2pt
\pinlabel $x$ [bl] at 76 248
\pinlabel $\delta_1$ [t] at 150 248
\pinlabel $x$ [tl] at 248 215
\pinlabel $\delta_2$ [r] at 248 148
\pinlabel $x$ [tr] at 212 43
\pinlabel $\delta_3$ [b] at 150 43
\pinlabel $x$ [br] at 43 81
\pinlabel $\delta_4$ [l] at 43 148
\pinlabel $E$ at 150 148
\endlabellist
\center
\includegraphics[scale=0.6]{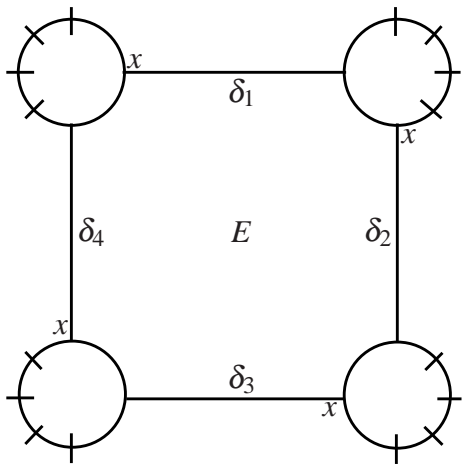}
%\scalebox{0.6}{\input{Schcycle.eps_t}}
\caption{A Scharlemann cycle of length 4 bounding an $a$--torsion $8$-gon.}
\label{Fig: Schcycle}
\end{figure}

\begin{lemma}[{\cite[Lemma 2.6.2]{CGLS}}] \label{Scharlemann exists}
If $\Lambda$ contains an $x$--cycle, then (possibly after a trivial 2--surgery on $D$), $\Lambda$ contains a Scharlemann cycle.  
\end{lemma}

\begin{proof}
The proof is again the same as in \cite{CGLS}. It uses the assumption that $Q$ is incompressible to eliminate circles of intersection on the interior of an innermost $x$--cycle.
\end{proof}

\begin{remark}
The presence of any such disc $D$ with $\Lambda$ containing a Scharlemann cycle is good enough for our purposes. So, henceforth, we assume that all circles in $\Lambda$ have been eliminated using the incompressibility of $Q$.
\end{remark}

\begin{remark}
In \cite{CGLS}, there is a distinction between $x$--cycles and, so-called, great $x$--cycles.  We do not need this here because all components of $D \cap F$ are parallel in $\eta(\boundary \alpha)$ as oriented curves.
\end{remark}

The next corollary explains the necessity of considering Scharlemann cycles.  

\begin{corollary}[\cite{CGLS}]\label{boundary edges}
If $\boundary D$ intersects fewer than $|\boundary Q \cap A|$ edges of $\Lambda$ then $\Lambda$ contains a Scharlemann cycle.
\end{corollary}

\begin{proof}
As $\boundary D$ contains fewer than $|\boundary Q \cap A|$ endpoints of boundary-edges in $\Lambda$ there is some $x \in \mc{C}$ which does not appear as a label on a boundary edge.  As every interior vertex of $\Lambda$ contains an edge with label $x$ at that vertex, none of those edges can be a boundary edge.  Consequently, $\Lambda$ satisfies $P(x)$.  Hence, by Lemmas \ref{P-condition} and \ref{Scharlemann exists}, $\Lambda$ contains a Scharlemann cycle of length $g$ (for some $g$).
\end{proof}

\begin{figure}[ht]
\labellist
\hair 2pt
\pinlabel $A$ [b] at 189.5 263
\pinlabel $R$ at 92 126
\pinlabel $x$ [l] at 305 145
\pinlabel $y$ [l] at 305 109
\endlabellist
\center
\includegraphics[scale=0.6]{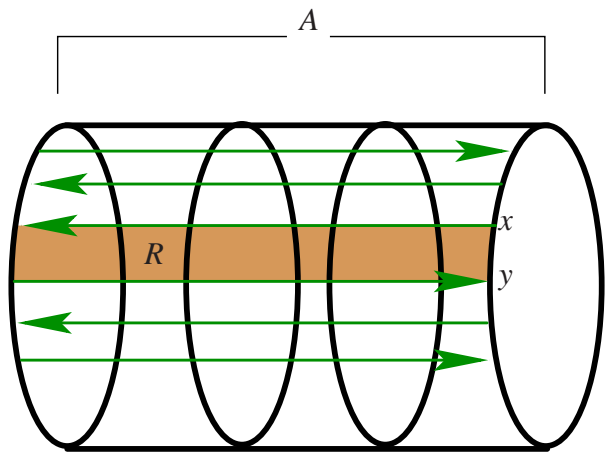}
%\scalebox{0.6}{\input{TorsionRectangle.eps_t}}
\caption{The rectangle $R$.}
\label{Fig: TorsionRectangle}
\end{figure}

In $A$ there is a rectangle $R$ with boundary consisting of the arcs $x$ and $y$ and subarcs of $\boundary A$. See Figure \ref{Fig: TorsionRectangle}. Because $\ob{\alpha}$ always intersects $D$ with the same sign, $\boundary D$ always crosses $R$ in the same direction. This shows that the arcs $\epsilon_i$ are all mutually parallel in $F$. The arcs $x$ and $y$ are antiparallel, so attaching $R$ to $Q$ produces an orientable surface. Hence, the interior of the Scharlemann cycle is an $a$--torsion $2g$--gon.
\end{proof}

We now proceed with proving the contrapositive of the theorem. Suppose that none of the three possible conclusions of the theorem hold.

The surface $Q$ is a parameterizing surface for the $\ob{\alpha}$--taut sutured manifold $(N[a],\gamma)$. Let 
\[
(N[a],\gamma) = (N_0,\gamma_0) \stackrel{S_1}{\to} (N_1,\gamma_1) \stackrel{S_2}{\to} \hdots \stackrel{S_{n}}{\to} (N_n,\gamma_n)\]
be an $\ob{\alpha}$--taut sutured manifold hierarchy for $(N[a],\gamma)$ which is adapted to $Q$. The surface $S_1$ may be obtained from the surface $S$ by performing the double-curve sum of $S$ with $k$ copies of $R_+$ and $l$ copies of $R_-$ (Theorem 2.6 of \cite{S3}).

The index $I(Q_i)$ is defined to be
\[
I(Q_i) = |\boundary Q_i \cap \boundary \eta(\ob{\alpha}_i)| + |\boundary Q_i \cap \gamma_i| - 2\chi(Q_i)
\]
where $Q_i$ is the parameterizing surface in $N_i$ and $\ob{\alpha}_i$ is the remnant of $\ob{\alpha}$ in $N_i$. Since $-2\chi(\ob{Q}) < K(\ob{Q})$, simple arithmetic shows that $I(Q) < 2|\boundary Q \cap \eta(a)|$. Since there is no $a$--torsion $2g$--gon for $Q$, by the previous proposition, there is no Gabai disc for $Q$. The proof of \cite[Theorem 9.1]{S3} shows that $(N_n,\gamma_n)$ is also $\nil$--taut, after substituting the assumption that there are no Gabai discs for $Q$ in $N$ wherever \cite[Lemma 9.3]{S3} was used (as in \cite[Proposition 4.1]{S4}).  In claims 3, 4, and 11 of \cite[Theorem 9.1]{S3} use the inequality $I(Q) < 2|\boundary Q \cap A|$ to derive a contradiction rather than the inequalities stated in the proofs of those claims.  

Hence, the hierarchy is $\nil$--taut, $(N[a],\gamma)$ is a $\nil$--taut sutured manifold and $S_1$ is a $\nil$--taut surface. Suppose that $S$ is not $\nil$--taut. Then there is a surface $S'$ with the same boundary as $S$ but with smaller Thurston norm. Then the double-curve sum of $S'$ with $k$ copies of $R_+$ and $l$ copies of $R_-$ has smaller Thurston norm than $S_1$, showing that $S_1$ is not $\nil$--taut. Hence, $S$ is $\nil$--taut.

The proof of \cite[Theorem 9.1]{S3} concludes by noting that at the final stage of the hierarchy, there is a cancelling or (non-self) amalgamating disc for each remnant of $\ob{\alpha}$.  When $N[a]$ is a solid torus the only $\nil$--taut conditioned surfaces are unions of discs.  If $S$ is chosen to be a single disc then $S_1$ is isotopic to $S$. To see this, notice that $R_\pm$ is an annulus and so the double-curve sum of $S$ with $R_\pm$ is isotopic to $S$. Hence, the hierarchy has length one and the cancelling and (non-self) amalgamating discs show that $\ob{\alpha}$ is braided in $N[a]$.
\qed
\begin{remark}
The proof proves more than the theorem states. It is actually shown that at the end of the hierarchy, $\ob{\alpha} \cap N_n$ consists of unknotted arcs in 3--balls. This may be useful in future work.
\end{remark}

For this theorem to be useful, we need to discuss the placement of sutures $\gamma$ on $\boundary N$ and the construction of a surface $Q$ without $a$--torsion $2g$--gons. The next two sections address these issues. In each of them, we restrict $F$ to being a genus two surface.
%%%%%%%%%%%%%%%% Placing Sutures %%%%%%%%%%%%%%%%%%%%%%
\section{Placing Sutures}\label{Sutures}
Let $N$ be a compact, orientable, irreducible 3--manifold with $F \subset \boundary N$ a component containing an essential simple closed curve $a$.
Suppose that $\boundary N - F$ is incompressible in $N$. For effective application of the main theorem, we need to choose curves $\gamma$ on $\boundary N[a]$ so that $(N[a],\gamma)$ is $\ob{\alpha}$--taut. With our applications in mind, we restrict our attention to the situation when the boundary component $F$ containing $a$ has genus 2. Define $\boundary_1 N[a] = \boundary N - F$ and $\boundary_0 N[a] = \boundary N[a] - \boundary_1 N[a]$.

For the moment, we consider only the choice of sutures $\hat{\gamma}$ on $\boundary_0 N[a]$.  If $a$ is separating, so that $\boundary_0 N[a]$ consists of two tori joined by the arc $\ob{\alpha}$, we do not place any sutures on $\boundary_0 N[a]$, i.e. $\hat{\gamma} = \nil$. (Figure \ref{suturechoice}.A.) If $a$ is non-separating, choose $\hat{\gamma}$ to be a pair of disjoint parallel loops on $F - \eta(a)$ which separate the endpoints of $\ob{\alpha}$. (Figure \ref{suturechoice}.B.) 

\begin{figure}[ht] 
\labellist
\hair 2pt
\pinlabel \textbf{A.} at 10 254
\pinlabel \textbf{B.} at 10 92
\pinlabel $\hat{\gamma}$ [r] at 158 56
\pinlabel $\ob{\alpha}$ [bl] at 294 132
\pinlabel $\ob{\alpha}$ [b] at 280.5 284
\endlabellist
\center
\includegraphics[scale=0.6]{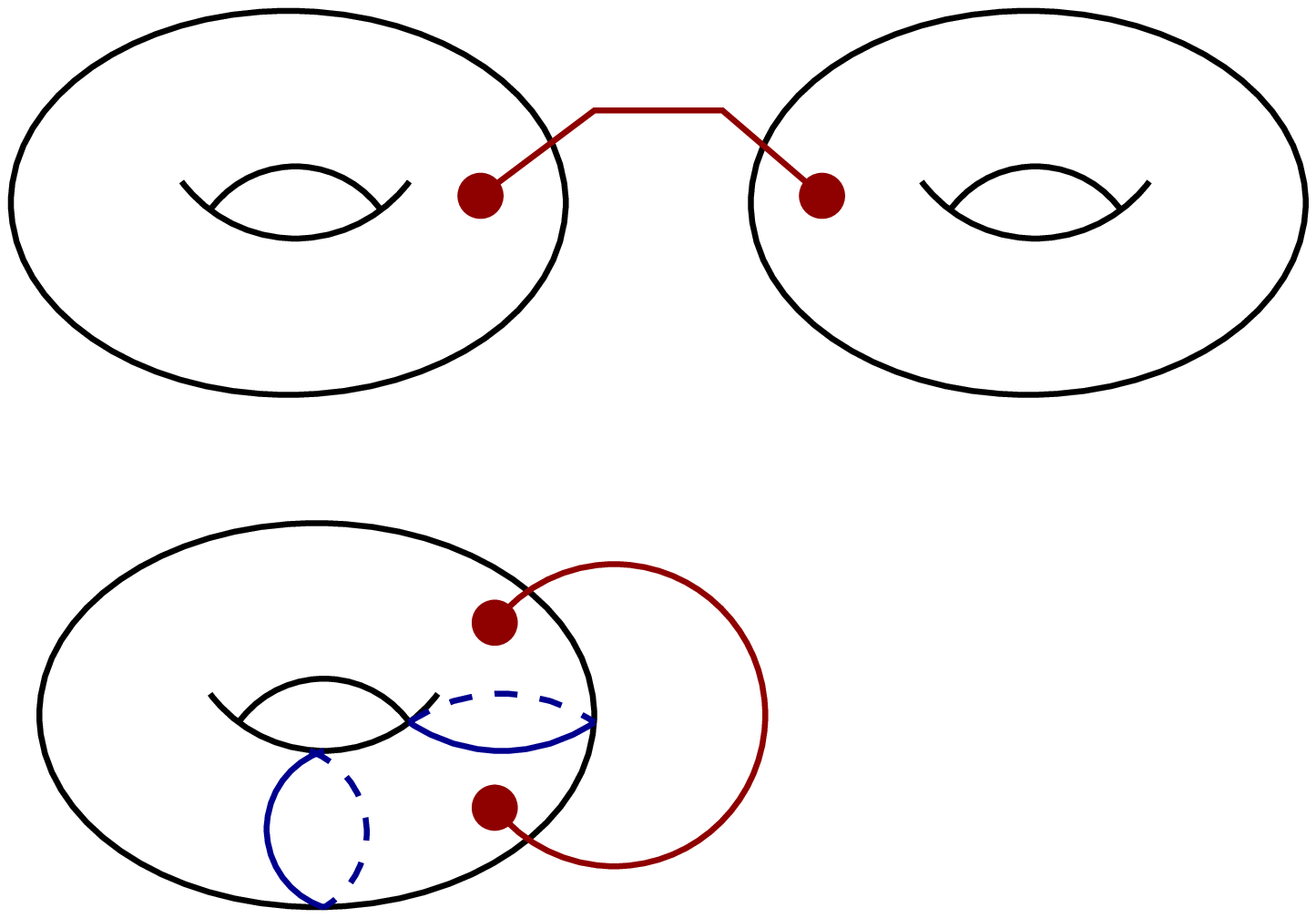}
%\scalebox{0.6}{\input{ChoosingSutures.eps_t}}
\caption{Choosing $\hat{\gamma}$.}
\label{suturechoice}
\end{figure}

If we are in the special situation of ``refilling meridians'', we will want to choose the curves $\hat{\gamma}$ more carefully. Recall that in this case $N \subset M$ and $F$ bounds a genus 2 handlebody $W \subset (M - \inter{N})$. The curves $a$ and $b$ bound in $W$ discs $\alpha$ and $\beta$ respectively. 

Assuming that the discs $\beta$ and $\alpha$ have been isotoped to intersect minimally and non-trivially, the intersection $\alpha \cap \beta$ is a collection of arcs. An arc of $\alpha \cap \beta$ which is outermost on $\beta$ cobounds with a subarc $\psi$ of $b$ a disc with interior disjoint from $\alpha$. This disc is a meridional disc of a (solid torus) component of $\boundary W - \inter{\eta}(\alpha)$. The arc $\psi$ has both endpoints on the same component of $\boundary \eta(a) \subset F$. We, therefore, define a \defn{meridional arc} of $b - a$ to be any arc of $b - \inter{\eta}(a)$ which together with an arc in $\boundary \eta(\alpha) \cap \inter{W}$ bounds a meridional disc of $W - \inter{\eta}(\alpha)$. If $a$ is non-separating, then the existence of meridional arcs shows that every arc of $b - \inter{\eta}(a)$ with endpoints on the same component of $\boundary \eta(a) \subset F$ is a meridional arc of $b - a$. An easy counting argument shows that if $a$ is non-separating then there are equal numbers of meridional arcs of $b - a$ based at each component of $\boundary \eta(a) \subset F$. Hence, when $a$ is non-separating, the number of meridional arcs of $b - a$, denoted $\mc{M}_a(b)$ is even. Some meridional arcs are depicted in Figure \ref{Meridional Arcs}.

\begin{figure}[ht]
\center
\includegraphics[scale=0.4]{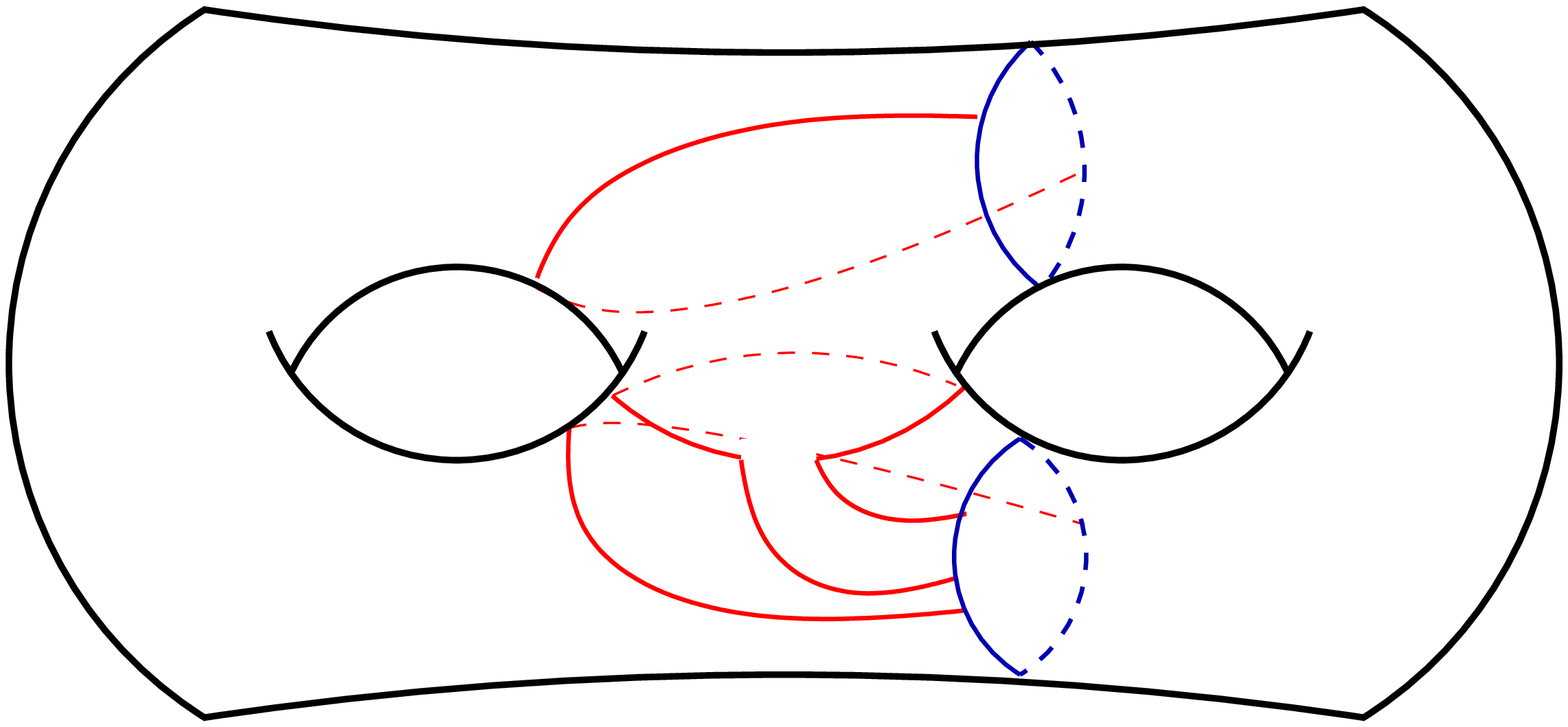}
%\scalebox{0.4}{\input{MeridianArcs.eps_t}}
\caption{Some meridional arcs on $\boundary W$}
\label{Meridional Arcs}
\end{figure}
Returning to the definition of the sutures $\hat{\gamma}$, we insist that when ``refilling meridians'' and when $\alpha$ is non-separating, the curves $\hat{\gamma}$ be meridional curves of the solid torus $W - \inter{\eta}(\alpha)$ which separate the endpoints of $\ob{\alpha}$ and which are disjoint from the meridional arcs of $b - a$ for a specified $b$.

We now show how to define sutures $\tild{\gamma}$ on non-torus components of $\boundary_1 N[a]$. Let $T(\gamma)$ be all the torus components of $\boundary_1 N[a]$. If $\boundary_1 N = T(\gamma)$ then $\tild{\gamma} = \nil$. Otherwise, the next lemma demonstrates how to choose $\tild{\gamma}$ so that, under certain hypotheses, $(N,\gamma \cup a)$ is taut, where $\gamma = \hat{\gamma} \cup \tild{\gamma}$.

\begin{lemma}\label{Lem: Choosing Sutures}
Suppose that $F - (\gamma \cup a)$ is incompressible in $N$. Suppose also that if $\boundary_1 N[a] \neq T(\gamma)$ then there is no essential annulus in $N$ with boundary on $\hat{\gamma} \cup a$. Then $\tild{\gamma}$ can be chosen so that $(N,\gamma \cup a)$ is $\nil$--taut and so that $(N[a],\gamma)$ is $\ob{\alpha}$--taut.  Furthermore, if $c \subset \boundary_1 N[a]$ is a collection of disjoint, non-parallel curves such that:
\begin{itemize}
\item $|c| \leq 2$
\item All components of $c$ are on the same component of $\boundary_1 N[a]$
\item No curve of $c$ cobounds an essential annulus in $N$ with a curve of $\hat{\gamma} \cup a$
\item If $|c| = 2$ then there is no essential annulus in $N$ with boundary $c$
\item If $|c| = 2$ and $a$ is separating, there is no essential thrice-punctured sphere in $N$ with boundary $c \cup a$.
\end{itemize}
then $\tild{\gamma}$ can be chosen to be disjoint from $c$.
\end{lemma}

The main ideas of the proof are contained in Section 5 of \cite{S4} and Theorem 2.1 of \cite{L2}. In \cite{S4}, Scharlemann considers ``special'' collections of curves on a non-torus component of $\boundary N$. These curves cut the component into thrice-punctured spheres. Exactly two of the curves in the collection bound once-punctured tori. In those tori are two curves of the collection which are called ``redundant''. The redundant curves are removed and the remaining curves form the desired sutures. Scharlemann shows how to construct such a special collection which is disjoint from a set of given curves and which gives rise to a taut-sutured manifold structure on the manifold under consideration. Lackenby, in \cite{L2}, uses essentially the same construction (but with fewer initial hypotheses) to construct a collection of curves cutting the non-torus components of $\boundary N$ into thrice-punctured spheres, but where all the curves are non-separating. We need to allow the sutures to contain separating curves as $c$ may contain separating curves. By slightly adapting Scharlemann's work, in the spirit of Lackenby, we can make do with the hypotheses of the lemma, which are slightly weaker than what a direct application of Scharlemann's work would allow.

\begin{proof}
Let $\tau$ be the number of once-punctured tori in $\boundary N$ with boundary some component of $c \cup a$. Since all components of $c$ are on the same component of $\boundary N$, $\tau \leq 4$ with $\tau \geq 3$ only if $a$ is separating.

Say that a collection of curves on $\boundary N$ is \defn{pantsless} if, whenever a thrice-punctured sphere has its boundary a subset of the collection, all components of the boundary are on the same component of $\boundary N$. If $a$ is non-separating, then $\tau \leq 2$. Hence, either $\tau \leq 2$ or $c \cup a \cup \hat{\gamma}$ is pantsless.

Scharlemann shows how to extend the set $c$ to a collection $\Gamma$, such that there is no essential annulus in $N$ with boundary on $
\Gamma \cup a \cup \hat{\gamma}$ and the curves $\Gamma$ cut $\boundary N$ into tori, once-punctured tori, and thrice-punctured spheres. Furthermore, if $c \cup a \cup \hat{\gamma}$ is pantsless, then so is $\Gamma \cup a \cup \hat{\gamma}$. An examination of Scharlemann's construction shows that all curves of $\Gamma - c$ may be taken to be non-separating. Thus, the number of once-punctured tori in $\boundary N$ with boundary on some component of $\Gamma \cup a$ is still $\tau$. If $\Gamma$ cannot be taken to be a collection of sutures on $\boundary N$, then, by construction, $|c| = 2$, one curve of $c$ bounds a once-punctured torus in $\boundary N$ containing the other curve of $c$. The component of $c$ in the once-punctured torus is ``redundant'' (in Scharlemann's terminology). If no curve of $c$ is redundant, let $\tild{\gamma} = \Gamma$; otherwise, form $\tild{\gamma}$ by removing the redundant curve from $\Gamma$. Let $\gamma' = \tild{\gamma} \cup a \cup \hat{\gamma}$. We now have a sutured manifold $(N,\gamma')$. Notice that the number of once-punctured torus components of $\boundary N - \gamma'$ is equal to $\tau$.

We now desire to show that $(N,\gamma')$ is $\nil$--taut. If it is not taut, then $R_\pm(\gamma)$ is not norm-minimizing in $H_2(N,\eta(\boundary R_\pm))$. Let $J$ be an essential surface in $N$ with $\boundary J = \boundary R_\pm = \gamma'$. Notice that $\chi_\nil(R_\pm) = -\chi(\boundary N)/2$ and that $|\gamma'| = -3\chi(\boundary N)/2 - \tau$.

Recall that either $\tau \leq 2$ or $\gamma'$ is pantsless. Suppose, first, that $\tau \leq 2$. Since no component of $J$ can be an essential annulus, by the arguments of Scharlemann and Lackenby, $\chi_\nil(J) \geq |\boundary J|/3 = |\gamma'|/3$. Hence, $\chi_\nil(J) \geq -\chi(\boundary N)/2 - \tau/3$. Since $\tau \leq 2$ and since $\chi_\nil(J)$ and $-\chi(\boundary N)/2$ are both integers, $\chi_\nil(J) \geq |\boundary N|/2 = \chi_\nil(R_\pm)$. Thus, when $\tau \leq 2$, $(N,\gamma')$ is a $\nil$--taut sutured manifold.

Suppose, therefore that $\gamma'$ is pantsless. Recall that $\tau \leq 4$. We first examine the case when each component of $J$ has its boundary contained on a single component of $\boundary M$. Let $J_0$ be all the components of $J$ with boundary on a single component $T$ of $\boundary N$. Let $\tau_0$ be the number of once-punctured torus components of $T - \gamma'$. Notice that $\tau_0 \leq 2$. The proof for the case when $\tau \leq 2$, shows that $\chi_\nil(J_0) \geq \chi_\nil(R_\pm \cap T)$. Summing over all component of $\boundary N$ shows that $\chi_\nil(J) \geq \chi_\nil(R_\pm)$, as desired.

We may, therefore, assume that some component $J_0$ of $J$ has boundary on at least two components of $\boundary N$. Since $\gamma'$ is pantsless, $\chi_\nil(J_0) \geq (|\boundary J_0| + 2)/3$. For the other components of $J$ we have, $\chi_\nil(J - J_0) \geq |\boundary(J-J_0)|/3$. Thus,
\[
\chi_\nil(J) \geq \frac{|\gamma'| + 2}{3} \geq -\frac{\chi(\boundary N)}{2} + \frac{2 - \tau}{3}
\]
Since $\tau \leq 4$ and since $\chi_\nil(J)$ and $-\chi(\boundary N)/2$ are both integers, we must have $\chi_\nil(J) \geq -\chi(\boundary N)/2 = \chi_\nil(R_\pm)$, as desired. Hence, $(N,\gamma') = (N,\gamma \cup a)$ is $\nil$--taut.
\end{proof}

\begin{remark}
The assumption that all components of $c$ are contained on the same component of $\boundary_1 N[a]$ can be weakened to a hypothesis on the number $\tau$. For what follows, however, our assumption suffices.
\end{remark}

We will be interested in when a component of $\boundary N - F$ becomes compressible upon attaching a 2--handle to $a \subset F$ and also becomes compressible upon attaching a 2--handle to $b \subset F$. If such occurs, the curves $c$ of the previous lemma will be the boundaries of the compressing discs for that component of $\boundary N$. Obviously, in order to apply the lemma we will need to make assumptions on how that component compresses.

%%%%%%%%%%%%%%%% CONSTRUCTING Q %%%%%%%%%%%%%%%%%%%%%%
\section{Constructing $Q$}\label{Constructing Q}
The typical way in which we will apply the main theorem is as follows. Suppose that $a$ and $b$ are simple closed curves on a genus two component $F \subset \boundary N$ and that there is an ``interesting'' surface $\ob{R} \subset N[b]$. We will want to use this surface to show that either $-2\chi(\ob{R}) \geq K(\ob{R})$ or $N[a]$ is taut. \textit{A priori}, though, the surface $R = \ob{R} \cap N$ may have $a$--boundary compressing discs or $a$--torsion $2g$--gons. The purpose of this section is to show how, given the surface $\ob{R}$ we can construct another surface $\ob{Q}$ which will, hopefully, have similar properties to $\ob{R}$ but be such that $Q = \ob{Q} \cap N$ does not have $a$--boundary compressing discs or $a$--torsion $2g$--gons. This goal will not be entirely achievable, but Theorem \ref{Thm: Constructing Q} shows how close we can come. Throughout we assume that $N$ is a compact, orientable, irreducible 3--manifold with $F \subset \boundary N$ a component having genus equal to 2. Let $a$ and $b$ be two essential simple closed curves on $F$ so that $a$ and $b$ intersect minimally and non-trivially. As before, let $\boundary_1 N = \boundary_1 N[b] = \boundary N - F$ and let $\boundary_0 N[b] = \boundary N[b] - \boundary_1 N[b]$. Let $T_0$ and $T_1$ be the components of $\boundary_0 N[b]$. If $b$ is non-separating, then $T_0 = T_1$.

Before stating the theorem, we make some important observations about $N[b]$ and surfaces in $N[b]$. If $b$ is non-separating, there are multiple ways to obtain a manifold homeomorphic to $N[b]$. Certainly, attaching a 2--handle to $b$ is one such way. If $b^*$ is any curve in $F$ which cobounds in $F$ with $\boundary \eta(b)$ a thrice-punctured sphere, then attaching 2--handles to both $b^*$ and $b$ creates a manifold with a spherical boundary component. Filling in that sphere with a 3--ball creates a manifold homeomorphic to $N[b]$. We will often think of $N[b]$ as obtained in the fashion. Say that a surface $\ob{Q} \subset N[b]$ is \defn{suitably embedded} if each component of $\boundary Q - \boundary \ob{Q}$ is a curve parallel to $b$ or to some $b^*$. We denote the number of components of $\boundary Q - \boundary \ob{Q}$ parallel to $b$ by $q = q(\ob{Q})$ and the number parallel to $b^*$ by $q^* = q^*(\ob{Q})$. Let $\tild{q} = q + q^*$. If $b$ is separating, define $b^* = \nil$. Define $\Delta = |b \cap a|$, $\Delta^* = |b^* \cap a|$, $\nu = |b \cap \gamma|$, and $\nu^* = |b^* \cap \gamma|$. We then have
\[
K(\ob{Q}) = (\Delta - \nu - 2)q + (\Delta^* - \nu^* - 2)q^* + \Delta_\boundary - \nu_\boundary.
\]

Define a \defn{slope} on a component of $\boundary N[b]$ to be an isotopy class of pairwise disjoint, pairwise non-parallel curves on that component. The set of curves is allowed to be the empty set. Place a partial order on the set of slopes on a component of $\boundary N[b]$ by declaring $r \leq s$ if there is some set of curves representing $r$  which is contained in a set of curves representing $s$. Notice that $\nil \leq r$ for every slope $r$. Say that a surface $\ob{R} \subset N[b]$ has boundary slope $\nil$ on a component of $\boundary N$ if $\boundary \ob{Q}$ is disjoint from that component. Say that a surface $\ob{R} \subset N[b]$ has boundary slope $r \neq \nil$ on a component of $\boundary N$ if each curve of $\boundary \ob{R}$ on that component is contained in some representative of $r$ and every curve of a representative of $r$ is isotopic to some component of $\boundary \ob{R}$.

Define a surface to be \defn{essential} if it is incompressible, boundary\hyph incompressible and has no component which is boundary-parallel or which is a 2--sphere bounding a 3--ball. The next theorem takes as input an essential surface $\ob{R} \subset N[b]$ and gives as output a surface $\ob{Q}$ such that $Q = \ob{Q} \cap N$ can (in many circumstances) be effectively used as a parameterizing surface. The remainder of the section will be spent proving it.

\begin{theorem}\label{Thm: Constructing Q}
Suppose that $\ob{R} \subset N[b]$ is a suitably embedded essential surface and suppose either
\begin{itemize}
\item[(I)] $\ob{R}$ is a collection of essential spheres and discs, or
\item[(II)] $N[b]$ contains no essential sphere or disc.
\end{itemize}
Then there is a suitably embedded incompressible and boundary\hyph incompressible surface $\ob{Q} \subset N[b]$ with the following properties. (The properties have been organized for convenience. The properties marked with a ``*'' are optional and need not be invoked.)
\begin{itemize}
\item[]
\item $\ob{Q}$ is no more complicated than $\ob{R}$:
	\begin{itemize}
	\item[(C1)] $(-\chi(\ob{Q}),\tild{q}(\ob{Q})) \leq (-\chi(\ob{R}),\tild{q}(\ob{R}))$ in lexicographic order
	\item[(C2)] The sum of the genera components of $\ob{Q}$ is no bigger than the sum of the genera of components of $\ob{R}$
	\item[(C3)] $\ob{Q}$ and $\ob{R}$ represent the same class in $H_2(N[b],\boundary N[b])$
	\end{itemize}
\item[]
\item The options for compressions, $a$--boundary compressions, and $a$--torsion $2g$--gons are limited:
	\begin{itemize}
\item[(D0)] $Q$ is incompressible.
\item[(D1)] Either there is no $a$--boundary compressing disc for $Q$ or $\tild{q} = 0$.
\item[(*D2)] If no component of $\ob{R}$ is separating and if $\tild{q} \neq 0$ then there is no $a$--torsion $2g$--gon for $Q$.
\item[(D3)] If $\ob{Q}$ is a disc or $2$--sphere then either $N[b]$ has a lens space connected summand or there is no $a$--torsion $2g$--gon for $Q$ with $g \geq 2$.
\item[(D4)] If $\ob{Q}$ is a planar surface then either there is no $a$--torsion $2g$--gon for $Q$ with $g \geq 2$ or attaching 2--handles to $\boundary N[b]$ along $\boundary \ob{Q}$ produces a 3--manifold with a lens space connected summand.
\end{itemize}
\item[]
\item The boundaries are not unrelated:
	\begin{itemize}
	\item[(*B1)] Suppose that (II) holds, that we are refilling meridians, that no component of $\ob{R}$ separates, and that $\boundary \ob{R}$ has exactly one non-meridional component on each component of $\boundary_0 N[b]$. Then $\ob{Q}$ has exactly one boundary component on each component of $\boundary_0 N[b]$ and the slopes are the same as those of $\boundary \ob{R} \cap \boundary_0 N[b]$.
	\item[(B2)] If $\boundary \ob{R} \cap \boundary_1 N$ is contained on torus components of $\boundary_1 N$ or if neither (D2) or (B1) are invoked, then the boundary slope of $\ob{Q}$ on a component of $\boundary_1 N[b]$ is less than or equal to the boundary slope of $\ob{R}$ on that component.
	\item[(B3)] If (D2) is not invoked and if the boundary slope of $\ob{R}$ on a component of $\boundary_0 N[b]$ is non-empty then the boundary slope of $\ob{Q}$ on that component is less than or equal to the boundary slope of $\ob{R}$.
	\end{itemize}
\end{itemize}
\end{theorem}

Property (B1), which is the most unpleasant to achieve, is present to guarantee that if $\ob{R}$ is a Seifert surface for $L_\beta$ then $\ob{Q}$ (possibly after discarding components) is a Seifert surface for $L_\beta$. This is not used subsequently in this paper, but future work is planned which will make use of it. However, achieving property (D2) which is used here, requires similar considerations. In this paper, we will often want to achieve (D2) which is incompatible with (B3). The paper \cite{T}, however, does not need (D2) and so we state the theorem in a fairly general form.

The only difficulty in proving the theorem is keeping track of the listed properties of $\ob{Q}$ and $\ob{R}$. Eliminating $a$--boundary compressions is psychologically easier than eliminating $a$--torsion $2g$--gons, so we first go through the argument that a surface $\ob{Q}$ exists which has all but properties (D2) - (D4). The argument may be easier to follow if, on a first reading, $\ob{R}$ is considered to be a sphere or essential disc. The proof is based on similar work in \cite{S5}, which restricts $\ob{R}$ to being a sphere or disc.  

The main purpose of assumptions (I) and (II) is to easily guarantee that the process for creating $\ob{Q}$ described below terminates. We will show that if $\tild{q}(\ob{R}) \neq 0$ and there is an $a$--boundary compressing disc or $a$--torsion $2g$-gon for $R = \ob{R} \cap N$ then there is a sequence of operations on $\ob{R}$ each of which reduces a certain complexity but preserves the properties listed above (including essentiality of $\ob{R}$). If (I) holds, the complexity is $(\tild{q}(\ob{R}),-\chi(R))$ (with lexicographic ordering). If (II) holds, the complexity is $(-\chi(\ob{R}),\tild{q}(\ob{R}))$ (with lexicographic ordering). If (II) holds, it is clear that $-\chi(\ob{R})$ is always non-negative. It will be evident that each measure of complexity has a minimum. The process stops either when $\tild{q} = 0$ or when the minimum complexity is reached.

\subsection{Eliminating compressions}
Suppose that $R$ is compressible and let $D$ be a compressing disc. Since $\ob{R}$ is incompressible, $\boundary D$ is inessential on $\ob{R}$. Compress $\ob{R}$ using $D$. Let $\ob{Q}$ be the new surface. $\ob{Q}$ consists of a surface of the same topological type as $\ob{R}$ and an additional sphere. We have $\tild{q}(\ob{Q}) = \tild{q}(\ob{R})$. If we are assuming (II), the sphere component must be inessential in $N[b]$ and so may be discarded. Notice that in both cases (I) and (II) the complexity has decreased. Since $R$ can be compressed only finitely many times, the complexity cannot be decreased arbitrarily far by compressions.

\subsection{Eliminating $a$--boundary compressions}
Assume that $\tild{q} \neq 0$ and that there is an $a$--boundary compressing disc $D$ for $R$ with $\boundary D = \delta \cup \epsilon$ where $\epsilon$ is a subarc of some essential circle in $\eta(a)$. There is no harm in considering $\epsilon \subset a - \boundary R$.

\subsubsection*{Case 1: $b$ separates $W$} 
In this case, $\eta(\ob{\beta}) - \interior \ob{R}$ consists $q - 1$ copies of $D^2 \times I$ labelled $W_1, \hdots, W_{q-1}$. There are two components $T_0$ and $T_1$ of $\boundary_0 N[b] = \boundary N[b] - \boundary N$, both tori. The frontiers of the $W_j$ in $\eta(\ob{\beta})$ are discs $\beta_1, \hdots, \beta_q$, each parallel to $\beta$, the core of the 2--handle attached to $b$.  Each 1-handle $W_j$ lies between $\beta_j$ and $\beta_{j+1}$.  The torus $T_0$ is incident to $\beta_1$ and the torus $T_1$ is incident to $\beta_q$.  See Figure \ref{WminusQ}.

\begin{figure}[ht] 
\labellist
\hair 2pt
\pinlabel $T_0$ [b] at 84 175
\pinlabel $W_1$ [b] at 144.5 175
\pinlabel $W_{q-1}$ [b] at 206 175
\pinlabel $T_1$ [b] at 274.5 175
\pinlabel $\beta_1$ [tl] at 135 24
\pinlabel $\beta_{q}$ [tl] at 256 24
\endlabellist
\center
\includegraphics[scale=0.8]{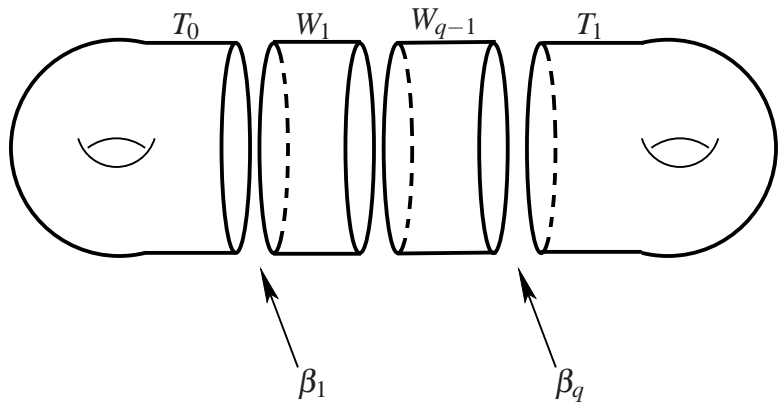}
%\scalebox{0.8}{\input{WminusQ.eps_t}}
\caption{The tori and 1-handles $W_j$}
\label{WminusQ}
\end{figure}

The interior of the arc $\epsilon \subset F$ is disjoint from $\boundary R$.  Consider the options for how $\epsilon$ could be positioned on $W$:

\subsubsection*{Case 1.1: $\epsilon$ lies in $\boundary W_j \cap F$ for some $1 \leq j \leq q-1$}
In this case, $\epsilon$ must span the annulus $\boundary W_j \cap F$.  The 1-handle $W_j$ can be viewed as a regular neighborhood of the arc $\epsilon$.  The disc $D$ can then be used to isotope $W_j$ through $\boundary D \cap R$ reducing $|\ob{R} \cap \ob{\beta}|$ by 2. See Figure \ref{Fig: Qisotopy1.1}. This maneuver decreases $\tild{q}(\ob{R})$. Alternatively, the disc $E$ describes an isotopy of $\ob{R}$ to a surface $\ob{Q}$ in $N[b]$ reducing $\tild{q}$. Clearly, $\ob{Q}$ satisfies the (C) and (B) properties. 

\begin{figure}[ht] 
\labellist
\hair 2pt
\pinlabel $R$ [t] at 102 250
\pinlabel $\delta$ [b] at 183 204
\pinlabel $W_j$ at 183 41.5
\endlabellist
\center
\includegraphics[scale=0.4]{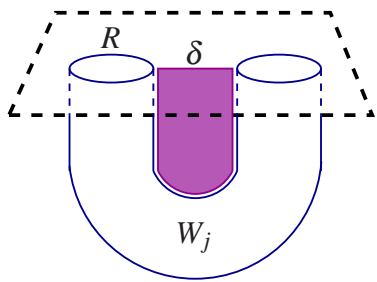}
%\scalebox{0.4}{\input{Qisotopy1pt1.eps_t}}
\caption{The disc $D$ describes an isotopy of $\ob{R}$.}
\label{Fig: Qisotopy1.1}
\end{figure}

Suppose, then, that $\epsilon$ is an arc on $T_0$ or $T_1$.  Without loss of generality, we may assume it is on $T_0$.

\subsubsection*{Case 1.2: $\epsilon$ lies in $T_0$ and has both endpoints on $\boundary \ob{R}$} This is impossible since $\ob{R}$ was assumed to be essential in $N[b]$ and $\tild{q} > 0$.
  
\subsubsection*{Case 1.3: $\epsilon$ lies in $T_0$ and has one endpoint on $\boundary \beta_1$ and the other on $\boundary \ob{R}$} The disc $D$ guides a proper isotopy of $\ob{R}$ to a surface $\ob{Q}$ in $N[b]$ which reduces $\tild{q}$. See Figure \ref{Fig: Qisotopy1.3}. Clearly, the (C) and (D) properties are satisfied.

\begin{figure}[ht] 
\labellist
\hair 2pt
\pinlabel $T_0$ [t] at 211 54
\pinlabel $\boundary \ob{R}$ [tr] at 67 35
\pinlabel $D$ [tl] at 421 110
\pinlabel $R$ [l] at 411 220
\pinlabel $\beta_1$ [b] at 404 284
\endlabellist
\center
\includegraphics[scale=0.4]{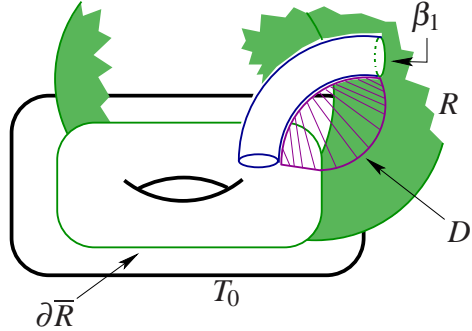}
%\scalebox{0.4}{\input{Qisotopy1pt3.eps_t}}
\caption{The disc $D$ describes an isotopy of $\ob{R}$.}
\label{Fig: Qisotopy1.3}
\end{figure}

\subsubsection*{Case 1.4: $\epsilon$ lies in $T_0$ and has endpoints on $\boundary \beta_1$.} Boundary-compressing $\ob{R} - \inter{\beta_1}$ produces a surface $\ob{J}$ with two new boundary components on $T_0$, both of which are essential curves. They are oppositely oriented and bound an annulus containing $\beta_1$. If $\boundary \ob{R} \cap T_0 \neq \nil$ then these two new components have the same slope on $T_0$ as $\boundary \ob{R}$, showing that property (B4) is satisfied. It is easy to check that $\chi(\ob{J}) = \chi(\ob{R})$ and that $\tild{q}(\ob{J}) = \tild{q}(\ob{R}) - 1$, so that (C1) is satisfied. Clearly, (C2), (C3), and (B3) are also satisfied.

If $\ob{J}$ were compressible, there would be a compressing disc for $\ob{R}$ by an outermost arc/innermost disc argument. Thus, $\ob{J}$ is incompressible. Suppose that $E$ is a boundary-compressing disc for $\ob{J}$ in $N[b]$ with $\boundary E = \kappa \cup \lambda$ where $\kappa$ is an arc in $\boundary N[b]$ and $\lambda$ is an arc in $\ob{J}$. Since $\ob{R}$ is boundary-incompressible, the arc $\kappa$ must lie on $T_0$ (and not on $T_1$). Since $T_0$ is a torus, either some component of $\ob{J}$ is a boundary-parallel annulus or $\ob{J}$ (and, therefore, $\ob{R}$) is compressible. We may assume the former. If $\ob{J}$ has other components apart from the boundary-parallel annulus, discarding the boundary-parallel annulus leaves a surface $\ob{Q}$ satisfying the (C) and (B) properties. We may, therefore, assume that $\ob{J}$ in its entirety is a boundary-parallel annulus.

Since $\chi(\ob{R}) = \chi(\ob{J})$, since $\ob{J}$ is a boundary-parallel annulus and since $\boundary \ob{J}$ has two more components then $\boundary \ob{R}$, $\ob{R}$ is an essential torus. However, using $D$ to isotope $\eta(\delta) \subset \ob{R}$ into $T_0$ and then isotoping $\ob{J}$ into $T_0$ gives a homotopy of $\ob{R}$ into $T_0$, showing that it is not essential, a contradiction. 

Thus, after possibly discarding a boundary-parallel annulus from $\ob{J}$ to obtain $\ob{L}$ we obtain a non-empty essential surface in $N[b]$ satisfying the first five required properties. If we do not desire property (B1) to be satisfied, take $\ob{Q} = \ob{L}$. Notice that this step may, for example, convert an essential sphere into two discs or an essential disc with boundary on $\boundary_1 N[b]$ into an annulus and a disc with boundary on $\boundary_0 N[b]$. This fact accounts for the delicate phrasing of the (B) properties.

Suppose, therefore, that we wish to satisfy (B1). Among other properties, we assume that $\ob{R}$ has a single boundary component on $T_0$.

There is an annulus $A \subset T_0$ which is disjoint from $\beta_1 \subset T_0$, which has interior disjoint from $\boundary \ob{L}$, and which has its boundary two of the two or three components of $\boundary \ob{L}$. See Figure \ref{Fig: AttachingAnnulus}. In the figure, the dashed line represents the arc $\epsilon$. The two circles formed by joining $\epsilon$ to $\boundary \beta_1$ are the two new boundary components of $\ob{L}$. Since, they came from a boundary-compression, they are oppositely oriented. If $\boundary \ob{R}$ has a single component on $T_0$ (indicated by the curve with arrows in the figure) then it must be oriented in the opposite direction from one of the new boundary components of $\boundary \ob{L}$. Attaching $A$ to $\ob{L}$ creates an orientable surface and does not increase negative euler characteristic or $\tild{q}$.

\begin{figure}[ht] 
\labellist
\hair 2pt
\pinlabel $T_0$ [t] at 47 41
\pinlabel $\beta_1$ at 398 189
\endlabellist
\center
\includegraphics[scale=0.4]{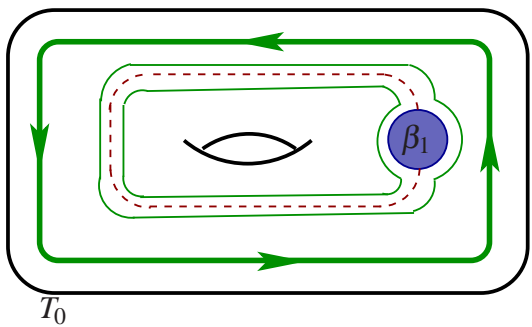}
%\scalebox{0.4}{\input{AttachingAnnulus.eps_t}}
\caption{The annulus $A$ lies between $\boundary \ob{R}$ and one of the new boundary components of $\ob{L}$.}
\label{Fig: AttachingAnnulus}
\end{figure}

Thus, if $|\boundary \ob{R} \cap T_0| \leq 1$, $\ob{L} \cup A$ is well-defined. It may, however, be compressible or boundary-compressible. Since it represents the homology class $[\ob{R}]$ in $H_2(N[b],\boundary N[b])$, as long as that class is non-zero we may thoroughly compress and boundary-compress it, obtaining a surface $\ob{J}$. Discard all null-homologous components of $\ob{J}$ to obtain a surface $\ob{Q}$. By assumption (II), we never discard an essential sphere or disc. Notice that since $\boundary \ob{R}$ has a single boundary component on $T_1$, the surface $\ob{Q}$ will also have a single boundary component on $T_1$. I.e. discarding separating components of $\ob{J}$ does not discard the component with boundary on $T_1$. Boundary-compressing $\ob{J}$ may change the slope of $\boundary \ob{J}$ on non-torus components of $\boundary_1 N[b]$. Discarding separating components may convert a slope on a torus component to the empty slope. Nevertheless, properties (B2) and (B3) still hold.

If a component of $\ob{J}$ is an inessential sphere then either $\ob{L}_A$ contained an inessential sphere or the sphere arose from compressions of $\ob{L}_A$. Suppose that the latter happened. Then after some compressions $\ob{L}_A$ contains a solid torus and compressing that torus creates a sphere component. Discarding the torus instead of the sphere shows that this process does not increase negative euler characteristic. If $\ob{L}_A$ contains an inessential sphere, this component is either a component of $\ob{L}$ and therefore of $\ob{R}$ or it arose by attaching $A$ to two disc components, $D_1$ and $D_2$, of $\ob{L}$. The first is forbidden by the assumption that $\ob{R}$ is essential and the second by (II). Consequently, negative euler characteristic is not increased.

Notice that, in general, compressing $\ob{L}_A$ may increase $\tild{q}$, but because $-\chi(\ob{Q})$ is decreased, property (C1) is still preserved and complexity is decreased. Since we assume (II) for the maneuvre, if (I) holds at the end of this case we can still conclude that $\tild{q}$ was decreased. (This is an observation needed to show that the construction of $\ob{Q}$ for the conclusion of the theorem terminates.)

\subsubsection*{Case 2: $b$ is non-separating and $q^* \neq 0$.}  
This is very similar to Case 1.  In what follows only the major differences are highlighted. 

Since $q^* \neq 0$, the cocore $\ob{\beta}^*$ of the 2--handle attached to $b^*$ and the cocore $\ob{\beta}$ form an arc with a loop at one end. Let $U = \eta(\ob{\beta^*} \cup \ob{\beta})$. Then $U- \ob{R}$ consists of a solid torus $q^* - 1$ copies of $D^2 \times I$ labelled $W^*_1, \hdots, W^*_{q^* - 1}$ with frontier in $U$ consisting of discs $\beta^*_1, \hdots, \beta^*_{q^*}$ parallel to  $\beta^*$ (the core of the 2--handle attached to $b^*$), a 3--ball $\mc{P}$ with frontier in $U$ consisting of 3 discs: $\beta^*_{q*}$, $\beta_1$, and $\beta_q$, $q - 1$ copies of $D^2 \times I$ labelled $W_1, \hdots, W_{q-1}$ with frontiers $\beta_1, \hdots, \beta_q$ consisting of discs parallel to $\beta$.  See Figure \ref{WminusQ2}. $\boundary_0 N[b]$ consists of a single torus $T_0$.

\begin{figure}[ht] 
\labellist
\hair 2pt
\pinlabel $T_0$ at 88 63
\pinlabel $\beta^*_1$ [t] at 117 50
\pinlabel $W^*_1$ at 149.5 63
\pinlabel $W^*_{q^* - 1}$ at 208 63
\pinlabel $\mc{P}$ at 287.5 63
\pinlabel $\beta_q$ [t] at 344 28
\pinlabel $W_{q-1}$ [t] at 378 59
\pinlabel $W_2$ [l] at 449 190
\pinlabel $\beta_2$ [l] at 463 228
\pinlabel $W_1$ [bl] at 394 249
\pinlabel $\beta_1$ [b] at 330 257
\endlabellist
\center
\includegraphics[scale=0.7]{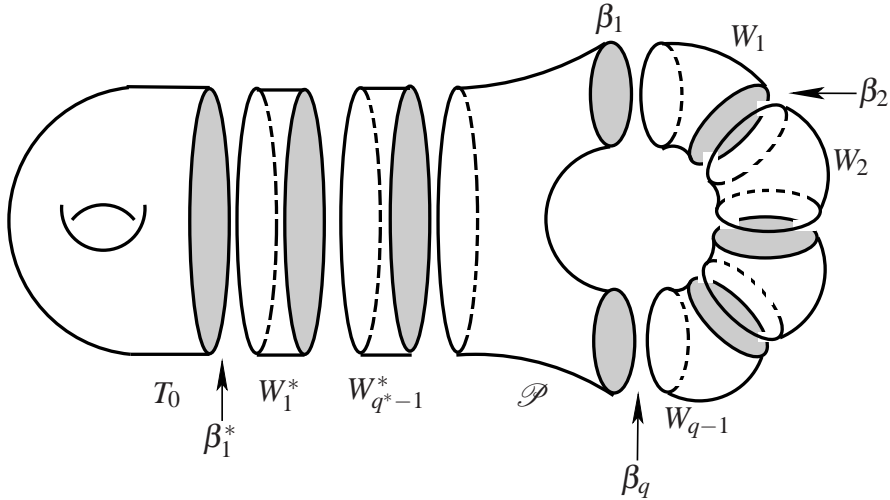}
%\scalebox{0.8}{\input{WminusQ2.eps_t}}
\caption{The torus, pair of pants, and 1-handles.}
\label{WminusQ2}
\end{figure}

\subsubsection*{Case 2.1 : $\epsilon$ is not located in $\mc{P}$} This is nearly identical to Case 1. To achieve (B1), an ``annulus attachment'' trick like that in Case 1.4 is necessary. 

\subsubsection*{Case 2.2: $\epsilon$ is located in $\mc{P}$} Since $\boundary \ob{R}$ is essential in $N[b]$ and since $\ob{R}$ is embedded, $\boundary \ob{R}$ is disjoint from $\mc{P}$.  The arc $\epsilon$ has its endpoints on exactly two of $\{\boundary \beta^*_{q^*}, \boundary \beta_1, \boundary \beta_{q}\}$.  Denote by $x$ and $y$ the two discs containing $\boundary \epsilon$ and denote the third by $z$.  That is, $\{\boundary x, \boundary y, \boundary z\} = \{\boundary \beta^*_{q^*}, \boundary \beta_1, \boundary \beta_{q}\}$.  Boundary-compressing $\cls(\ob{Q} - (x \cup y))$ along $D$ removes $\boundary x$ and $\boundary y$ as boundary-components of $R$ and adds another boundary-component parallel to $\boundary z$.  Attach a disc in $F$ parallel to $z$ to this new component, forming $\ob{J}$. $\ob{J}$ is isotopic in $N[b]$ to $\ob{R}$ (Figure \ref{Fig: PantsIsotopy}) and is, therefore, essential and satisfies the (C) and (B) properties.

\begin{figure}[ht] 
\labellist
\hair 2pt
\pinlabel $z$ at 19 206
\pinlabel $x$ at 217 278
\pinlabel $y$ at 155 133
\pinlabel $\mc{P}$ [b] at 68 267
\pinlabel $D$ at 139 215
\pinlabel $R$ [br] at 235 400
\endlabellist
\center
\includegraphics[scale=0.5]{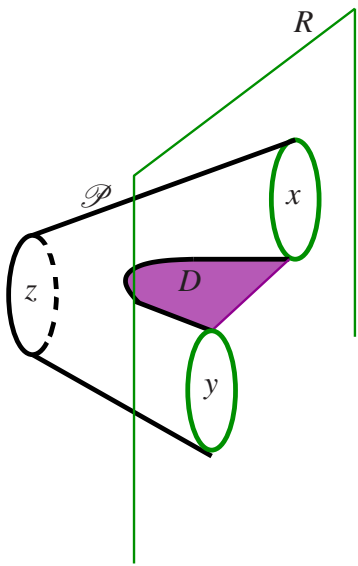}
%\scalebox{0.6}{\input{PantsIsotopy.eps_t}}
\caption{The disc $D$ in Case 2.2}
\label{Fig: PantsIsotopy}
\end{figure}

\subsubsection*{Case 3: $b$ is non-separating and $q^* = 0$} 

Since $b$ is non-separating, $\eta(\ob{\beta}) - \ob{Q}$ consists of copies of $D^2 \times I$ labelled $W_1, \hdots, W_{q-1}$ which are separated by discs $\beta_1, \hdots, \beta_q$ each parallel to $\beta$ so that each $W_i$ is adjacent to $\beta_{i}$ and $\beta_{i+1}$ where the indices run mod $q$. $\boundary_0 N[b]$ is a single torus $T_0$. See Figure \ref{WminusQnonsep}.

\begin{figure}[ht] 
\labellist
\hair 2pt
\pinlabel $T_0$ [t] at 91.5 72
\pinlabel $\beta_q$ [t] at 265 24
\pinlabel $W_{q-1}$ [t] at 301 56
\pinlabel $W_2$ [l] at 369 184
\pinlabel $W_1$ [bl] at 309 245
\pinlabel $\beta_1$ [b] at 251 251
\pinlabel $\beta_2$ [l] at 384 224
\endlabellist
\center
\includegraphics[scale=0.7]{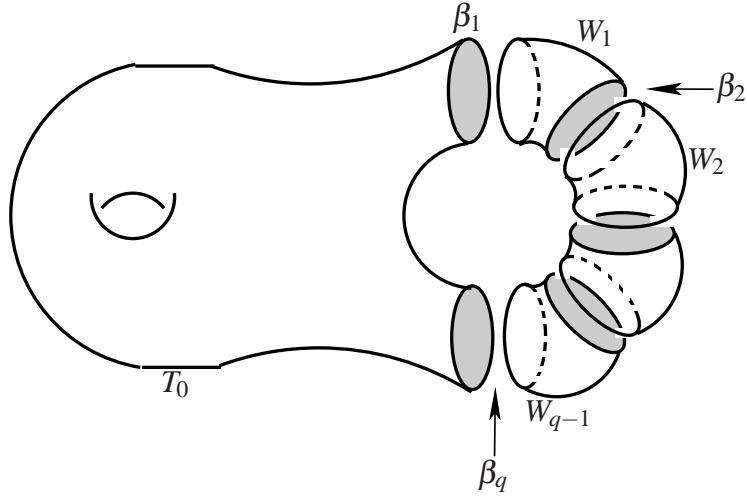}
%\scalebox{0.8}{\input{WminusQnonsep.eps_t}}
\caption{The solid torus and 1-handles $W_j$}
\label{WminusQnonsep}
\end{figure}

We need only consider the following cases, as the others are similar to prior cases.

\subsubsection*{Case 3.4: $\epsilon$ is located on $T_0$ and either both endpoints are on $\boundary \beta_1$ or both are on $\boundary \beta_{q}$}  The arc $\epsilon$ is a meridional arc.  Suppose, without loss of generality, that $\boundary \epsilon \subset \boundary \beta_1$. Boundary-compress $\ob{R} - \inter{\beta_1}$ along $D$.  This creates a surface $\ob{J}$ with boundary on $T_0$. After possibly discarding a boundary-parallel annulus $\ob{J}$ is essential and the (C) properties hold as well as (B2) and (B3). We need to show that (B1) can be achieved, if desired.

Suppose that we are in the situation of ``refilling meridians'' so that $N \subset M$ and $F$ bounds a genus 2 handlebody $W$ in $M - N$ with $a$ and $b$ bounding discs in $W$. Then since the endpoints of $\epsilon$ are on the same component of $\boundary \eta(a) \subset F$, $\epsilon$ is a meridional arc of $b - a$. If $\boundary \ob{R}$ is not meridional on $T_0$ this case, therefore, cannot occur. Thus, the (C) and (B) properties hold.

\subsubsection*{Case 3.5: $\epsilon$ is located on $T_0$ and has one endpoint on $\beta_1$ and the other on $\beta_q$} The disc $D$ guides an isotopy of $\ob{R}$ to a surface $\ob{Q}$ which is suitably embedded in $M[\beta]$ and has $q^*(\ob{Q})  = 1$. We have $\tild{q}(\ob{Q}) = \tild{q}(\ob{R}) - 1$. The surface $\ob{Q}$ can also be created by boundary-compressing $\ob{R} - (\beta_1 \cup \beta_q)$ with $D$ and then adding a disc $\beta^*$ to the new boundary component. See Figure \ref{Fig: PantsIsotopy}. Clearly, the (C) and (B) properties hold.

The previous cases have each described an operation on $\ob{R}$ which produces an essential surface $\ob{Q}$ having the (C) and (B) properties. Furthermore, the maneuvre described in each case strictly decreases complexity. Thus, after repeating the operation enough times either the surface $\ob{Q}$ will have $\tild{q}(\ob{Q}) = 0$ or there will be no $a$--boundary compressions for $Q$. That is, the (C) and (B) properties hold and, in addition, (D1) holds.

\subsection{Eliminating $a$--torsion $2g$--gons}
We may now assume that there is an $a$--torsion $2g$--gon $D$ for $Q$ with $g \geq 2$ (since an $a$--torsion $2$--gon is an $a$--boundary compressing disc). For ease of notation, relabel and let $\ob{R} = \ob{Q}$ and $R = Q$. By the definition of $a$--torsion $2g$--gon, there is a rectangle $E$ containing the parallel arcs $\boundary D \cap F$ which, when attached to $R$, creates an orientable surface. Two opposite edges of $\boundary E$ lie on $\boundary R$ and the other two are parallel (as un-oriented arcs) to the arcs of $\boundary D \cap F$. Denote the components of $\boundary R$ containing the two edges of $\boundary E$ by $\boundary_x$ and $\boundary_y$. It is entirely possible that $\boundary_x = \boundary_y$. If $\boundary_x$ is a component of $\boundary R - \boundary \ob{R}$, let $\beta_x$ denote the disc in $\ob{R} - R$ which it bounds. Similarly define $\beta_y$. 

Suppose that $\ob{R}$ is a planar surface or 2--sphere. Let $\wihat{N}$ be the 3--manifold obtained from $N[b]$ by attaching 2--handles to $\boundary N[b]$ in such a way that each component, but one, of $\boundary \ob{J}$ bounds a disc in $\wihat{N}$. Attach these discs to $\ob{R}$ forming a surface $\wihat{R}$. Since $\ob{R}$ was a planar surface or 2--sphere, $\wihat{R}$ is a disc or 2--sphere. A regular neighborhood of $\wihat{R} \cup E$ is a solid torus and the disc $D$ is in the exterior of that solid torus and winds longitudinally around it $n \geq 2$ times. Thus $\eta(\wihat{R} \cup E \cup D)$ is a lens space connected summand of $\wihat{N}$. Hence, redefining $\ob{Q} = \ob{J}$ we satisfy the (C), (B), and (D) properties.

We may, therefore, assume that $\ob{R}$ is not a planar surface or 2--sphere. We need to show that we can achieve (D2) in addition to the (C), (B), and (D1) properties. The surface $\ob{R}' = (\ob{R} - (\beta_x \cup \beta_y)) \cup E$ is compressible by the disc $D$. Compress it to obtain an orientable surface $\ob{J}$. Notice that 
\[
 (-\chi(\ob{J}),\tild{q}(\ob{J})) < (-\chi(\ob{R}),\tild{q}(\ob{R})).
\]

Analyzing the position of $E$ as we did the position of $\epsilon$ in the previous section and possibly performing the ``annulus attachment trick'', we can guarantee that the (C) and (B) properties are satisfied. If the ends of $E$ are both on $\boundary \ob{R}$ then the boundary of $\ob{J}$ may have different slope from the boundary of $\ob{R}$. Whether or not we perform the annulus attachment trick, the surface $\ob{J}$ may be inessential. Compressing, boundary\hyph compressing, and discarding null-homologous components produces a non-empty essential surface $\ob{Q}$ satisfying properties (B) and (C). Considerations similar to those necessary for achieving (B1) in case 1.4 explain why (B2) is phrased as it is. (B3) is incompatible with (D2) since discarding components may discard $\boundary \ob{R} \cap \boundary_0 N[b]$ converting a non-empty slope to an empty slope. A future attempt to eliminate an $a$--boundary compressing disc or $a$--torsion $2g$--gon may then introduce new boundary components on $\boundary_0 N[b]$ of different slope. 

As before, complexity has been strictly decreased for both assumptions (I) and (II). Of course, we may now have additional compressing discs, $a$--boundary compressing discs, or $a$--torsion $2g$--gons to eliminate as in the previous sections. Since all these operations lower complexity, the process terminates with the required surface $\ob{Q}$.
\qed

The surface $\ob{Q}$ produced by the previous theorem may be disconnected. (For example, if $b$ is separating it is possible we could start with $\ob{R}$ being a disc with boundary on $T_0$ and end up with $\ob{Q}$ the union of an annulus with boundary on $T_0 \cup T_1$ and a disc with boundary on $T_1$.) The next corollary puts our minds at rest by elucidating when we can discard components to arrive at a connected surface $\ob{Q}$.

\begin{corollary}\label{Cor: Constructing Q}
The following statements are true:
\begin{itemize}
\item If $\ob{R}$ is a collection of spheres or discs then after discarding components of the surface $\ob{Q}$ created by Theorem \ref{Thm: Constructing Q} we may assume that $\ob{Q}$ is an essential sphere or disc such that $\tild{q}(\ob{Q}) \leq \tild{q}(\ob{R})$ and conclusions (B2), (B3), (D0), (D1), (D3), and (D4) hold. 
\item If $N[b]$ does not contain an essential disc or sphere, then we may assume the $\ob{Q}$ produced by Theorem \ref{Thm: Constructing Q} to be connected and Conclusions (C1), (C2), (B2), and (D0) - (D4) hold. Furthermore, if $\ob{R}$ is non-separating, so is $\ob{Q}$. 
\end{itemize}
\end{corollary}
\begin{proof}
Suppose that $\ob{R}$ is a collection of spheres or a discs and let $\tild{Q}$ be the surface produced by Theorem \ref{Thm: Constructing Q}. Notice that each component of $\tild{Q} \cap N$ is incompressible. Since $-\chi(\ob{R}) < 0$, by conclusions (C1) and (C2) of that theorem, $-\chi(\tild{Q}) < 0$ and each component of $\tild{Q}$ is a planar surface or $\tild{Q}$ is a sphere. Indeed, at least one component $\ob{Q}$ of $\ob{Q}$ is a sphere or disc. By conclusion (D1), either $\tild{Q}$ is disjoint from $\ob{\beta}$ or there is no $a$--boundary compressing disc for $\tild{Q} \cap N$. If there is an $a$--boundary compressing disc for $\ob{Q} \cap N$ then an outermost arc argument shows that there would be one for $\tild{Q} \cap N$. Thus, either $\ob{Q}$ is disjoint from $\ob{\beta}$ or there is no $a$--boundary compressing disc for $\ob{Q}$. As argued in the proof of Theorem \ref{Thm: Constructing Q}, if there is an $a$--torsion $2g$--gon for $Q$, then $N[b]$ contains a lens-space connected summand. It is clear, therefore, that the required conclusions hold.

Suppose that $N[b]$ contains no essential disc or sphere. Let $\tild{Q}$ be the surface produced by Theorem \ref{Thm: Constructing Q} and notice that $\tild{Q}$ contains no disc or sphere components. Notice also that each component of $\tild{Q} \cap N$ is incompressible. Choose a component $\tild{Q}_0$ of $\tild{Q}$ and discard the other components. Neither negative euler charactistic nor $\tild{q}$ are raised. If $\ob{R}$ was non-separating, choose $\tild{Q}_0$ to be non-separating. Either $\tild{Q}_0$ satisfies the conclusion of the corollary or $\tild{q}(\tild{Q}_0) > 0$ and there is an $a$--boundary compressing disc or $a$--torsion $2g$--gon for $ \tild{Q}_0 \cap N$. Apply the theorem with $\ob{R} = \tild{Q}_0$ and notice that the surface $\tild{Q}_1$ produced has strictly smaller complexity. Thus, repeating this process, each time discarding all but one component, we eventually obtain the connected surface $\ob{Q}$ promised by corollary. 
\end{proof}

\section{Refilling Meridians}\label{RM}
We now turn to applying the Main Theorem to ``refilling meridians''. For the remainder, suppose that $M$ is a 3--manifold containing an embedded genus 2 handlebody $W$. Let $N = M - \inter{W}$. Let $\alpha$ and $\beta$ be two essential discs in $W$ isotoped to intersect minimally and non-trivially. Let $a = \boundary \alpha$, $b = \boundary \beta$, $b^* = \boundary \beta^*$, $M[\alpha] = N[a]$, and $M[\beta] = N[b]$. Recall that $L_\alpha$ and $L_\beta$ are the cores of the solid tori produced by cutting $W$ along $\alpha$ and $\beta$ respectively. If we need to place sutures $\hat{\gamma}$ on $F = \boundary W$ we will do so as described in Section \ref{Sutures}. We begin by briefly observing that for any suitably embedded surface $\ob{Q} \subset M[\beta]$, with boundary disjoint from $\gamma \cap \boundary M$, $K(\ob{Q}) \geq 0$.

If $\alpha$ is separating, 
\[
K(\ob{Q}) = q(\Delta - 2) + q^*(\Delta^* - 2) + \Delta_\boundary.
\]
Since $b$, $b^*$, and $a$ all bound discs in $W$, $\Delta$ is at least two. If $q^* \neq 0$, then $\Delta^*$ is also at least two. Thus, $K(\ob{Q}) \geq 0$.

Recall (Section \ref{Sutures}) that if $\alpha$ is non-separating, any arc of $b - \inter{\eta}(a)$ with endpoints on the same component of $\boundary \eta(a)$ is a meridional arc of $b-a$. The number of these meridional arcs is denoted $\mc{M}_a(b)$ and it is always even and always at least two since there are the same number of meridional arcs based at each component of $\boundary \eta(a) \subset F$. The sutures $\hat{\gamma}$ are disjoint from these meridional arcs. Since any arc of $b-a$ which is not a meridional arc intersects exactly one suture exactly once, we have
\[
\Delta - \nu = \mc{M}_a(b) \geq 2
\]
and
\[
\Delta^* - \nu^* \geq \mc{M}_a(b^*) \geq 2.
\]
Since $\boundary \ob{Q}$ is disjoint from $b \cup b^*$, it is also disjoint from the meridional arcs of $b - a$. Consequently, each arc of $\boundary \ob{Q} - a$ intersects $\hat{\gamma}$ at most once. Hence, $\Delta_\boundary - \nu_\boundary \geq 0$. When $\alpha$ is non-separating, we, therefore, have
\[
K(\ob{Q}) \geq q(\mc{M}_a(b) - 2) + q^*(\mc{M}_a(b^*) - 2) + \Delta_\boundary - \nu_\boundary \geq 0.
\]
\subsection{Scharlemann's Conjecture}\label{RM: Scharlemann}
Studying the operation of refilling meridians, Scharlemann \cite{S5} was led to the following definitions and conjecture.

Define $(M,W)$ to be \defn{admissible} if
\begin{itemize}
\item[(A0)] every sphere in $M$ separates
\item[(A1)] $M$ contains no lens space connected summands
\item[(A2)] Any two curves in $\boundary M$ which compress in $M$ are isotopic in $\boundary M$
\item[(A3)] $M - W$ is irreducible
\item[(A4)] $\boundary M$ is incompressible in $N$.
\end{itemize}

He conjectured
\begin{conjecture}
If $(M,W)$ is admissible then one of the following occurs
\begin{itemize}
\item $M = S^3$ and $W$ is unknotted (i.e. $N$ is a handlebody)
\item At least one of $M[\alpha]$ and $M[\beta]$ is irreducible and boundary-irreducible
\item $\alpha$ and $\beta$ are ``aligned'' in $W$.
\end{itemize}
\end{conjecture}

The definition of ``aligned'' is rather complicated and is not needed for what follows, so I will not define it here. I will only remark that it is a notion which is independent of the embedding of $W$ in $M$.

Scharlemann proved the following for admissible pairs $(M,W)$:
\begin{theorem*}[Scharlemann]
\begin{itemize}
\item[]
\item If $\boundary W$ compresses in $N$ then the conjecture is true.
\item If $\Delta \leq 4$ then the conjecture is true.
\item If $\alpha$ is separating and $M$ contains no summand which is a non-trivial rational homology sphere then one of $M[\alpha]$ and $M[\beta]$ is irreducible and boundary-irreducible.
\item If both $\alpha$ and $\beta$ are separating then the conjecture is true. If, in addition, $\Delta \geq 6$, then one of $M[\alpha]$ and $M[\beta]$ is irreducible and boundary-irreducible.
\end{itemize}
\end{theorem*}

With a slight variation on the notion of ``admissible'', Scharlemann's Conjecture can now be completed for a large class of manifolds.

Define the pair $(M,W)$ to be \defn{licit} if the following hold:
\begin{itemize}
\item[(L0)] $H_2(M) = 0$.
\item[(L1)] $H_1(M)$ is torsion-free.
\item[(L2)] No curve on a non-torus component of $\boundary M$ which compresses in $M$ bounds an essential annulus in $N$ with a meridional curve of $\boundary W$ (that is, a curve on $\boundary W$ which bounds a disc in $W$).
\item[(L3)] $N$ is irreducible
\item[(L4)] $\boundary M$ is incompressible in $N$.
\end{itemize}

The major improvement provided by the next theorem, is that the case of non-separating meridians can be effectively dealt with. The theorem nearly completes Scharlemann's conjecture for pairs $(M,W)$ which are both licit and admissible. The one major aspect of Scharlemann's conjecture which is not covered by this theorem is the question of whether or not both of $M[\alpha]$ and $M[\beta]$ can be solid tori. In \cite{T}, this case is resolved.

\begin{theorem}[Modified Scharlemann Conjecture]\label{Thm: MSC}
Suppose that $(M,W)$ is licit and that $\alpha$ and $\beta$ are two essential discs in $W$. Suppose that $\boundary W$ is incompressible in $N$. Then either $\alpha$ and $\beta$ can be isotoped to be disjoint or all of the following hold:
\begin{itemize}
\item One of $M[\alpha]$ or $M[\beta]$ is irreducible
\item If one of $M[\alpha]$ or $M[\beta]$ is reducible then no curve on $\boundary M$ compresses in the other.
\item No curve on $\boundary M$ compresses in both $M[\alpha]$ and $M[\beta]$.
\item If one of $M[\alpha]$ or $M[\beta]$ is a solid torus, then the other is not reducible.
\end{itemize}
\end{theorem}

Conditions (L0) and (L1) are stronger than Conditions (A0) and (A1) but are used to guarantee that $H_1(M[\alpha])$ and $H_1(M[\beta])$ are torsion-free; this is required for the application of the main theorem. Condition (L2) is neither stronger nor weaker than Condition (A2) since we allow multiple curves on $\boundary M$ to compress in $M$ but forbid the existence of certain annuli. To show that some condition like (A2) was required, Scharlemann points out the following example:
\begin{example}
Let $M$ be a genus 2--handlebody and let $W \subset M$ so that $M - \inter{W}$ is a collar on $\boundary W$. (That is, $M$ is a regular neighborhood of $W$.) Then conditions (A0), (A1), (A3), (A4), (L0), (L1), (L3), and (L4) are all satisfied. But given any essential disc $\alpha \subset W$, $M[\alpha]$ is obviously boundary-reducible. Both (A2) and (L2) rule out this example.
\end{example}

The Modified Scharlemann Conjecture is simply a ``symmetrized'' version of the following theorem. In this theorem, the incompressiblity assumption has been weakened for later applications.

\begin{theorem}\label{Thm: MSC Unsymmetrized}
Suppose that $(M,W)$ is licit and that $\alpha$ and $\beta$ are two essential discs, isotoped to intersect minimally, with $\Delta > 0$. Suppose that $M[\beta]$ is reducible or boundary-reducible. If $\alpha$ is separating, assume that $\boundary W - a$ is incompressible in $N$. If $\beta$ is non-separating, assume that there is no essential disc in $M[\beta]$ which is disjoint from both $\ob{\beta}$ and $a$. Then the following hold:
\begin{itemize}
\item $M[\alpha]$ is irreducible
\item If $M[\beta]$ is reducible, no essential curve in $\boundary M$ compresses in $M[\alpha]$
\item If $M[\beta]$ is boundary-reducible, no essential curve of $\boundary M$ compresses in both $M[\beta]$ and $\boundary M[\beta]$.
\end{itemize}
\end{theorem}

\begin{proof}
We begin by showing that $H_1(M[\alpha])$ is torsion-free. Consider $M$ as the union of $V = W - \inter{\eta}(\alpha)$ and $M[\alpha]$. Using assumption (L0) that $H_2(M) = 0$, we see that the Mayer-Vietoris sequence gives the exact sequence:
\[
0 \to H_1(\boundary V) \stackrel{\phi}{\to} H_1(M[\alpha]) \oplus H_1(V) \stackrel{\psi}{\to} H_1(M) \to 0.
\]

Suppose that $x$ is an element of $H_1(M[\alpha])$ and that $n \in \N$ is such that $nx = 0$. Then $n\psi(x,0) = \psi(nx,0) = 0$. Since $H_1(M)$ is torsion-free, $\psi(x,0) = 0$. Thus, by exactness, $(x,0)$ is in the image of $\phi$. Let $y \in H_1(\boundary V)$ be in the preimage of $(x,0)$. Also, $\phi(ny) = n\phi(y) = (nx,0) = (0,0)$. From exactness, we know that $\phi$ is injective. Hence, $ny = 0 \in H_1(\boundary V)$. The boundary of $V$ is a collection of tori and, therefore, $H_1(\boundary V)$ is torsion-free. Consequently, $y = 0$. Therefore, $x = 0$ and $H_1(M[\alpha])$ is torsion-free.

We now proceed with the theorem by choosing appropriate sutures on $\boundary M$. If $\boundary M$ is compressible in $M[\beta]$, let $c_\beta$ be a curve on $\boundary M$ which compresses in $M[\beta]$. If $c_\beta = \nil$, let $c$ be any curve on $\boundary M$ which compresses in $M$, otherwise let $c = c_\beta$.

By Lemma \ref{Lem: Choosing Sutures}, we may choose sutures $\gamma$ on $\boundary M[\alpha]$ so that $\hat{\gamma} = \gamma \cap \boundary_0 M[\alpha]$ is chosen as usual and so that $\gamma \cap c = \nil$ and $(M[\alpha],\gamma)$ is an $\ob{\alpha}$--taut sutured manifold. Let $\ob{R}$ be either an essential sphere, an essential disc with boundary $c_\beta = c$, or an essential disc with boundary on $\boundary_0 M[\beta]$. Let $\ob{Q}$ be the result of applying Corollary \ref{Cor: Constructing Q} to $\ob{R}$. $\ob{Q}$ is an essential sphere, an essential disc with boundary $c_\beta$, or an essential disc with boundary on $\boundary_0 M[\beta]$. 

If $\ob{Q}$ is a sphere or disc with boundary $c_\beta$, then, since $N$ is irreducible and $\boundary M$ is incompressible in $N$, $\tild{q} > 0$. By Corollary \ref{Cor: Constructing Q}, there is no compressing disc, $a$--boundary compressing disc, or $a$--torsion $2g$--gon for $Q = \ob{Q} \cap N$. Suppose, for the moment, that $\ob{Q}$ is a disc with boundary on $\boundary W$. If $\tild{q} > 0$ then $Q$ is not disjoint from $a$. By Corollary \ref{Cor: Constructing Q} there is no compressing disc, $a$--boundary compressing disc or $a$--torsion $2g$--gon for $Q$. If $\tild{q} = 0$ then by hypothesis $Q = \ob{Q}$ is not disjoint from $a$. Since $Q = \ob{Q}$ is a disc, there are no essential arcs in $Q$ and so there is no compressing disc, $a$--boundary compressing disc, or $a$--torsion $2g$--gon in this case either. 

Since, in all cases, $\boundary \ob{Q}$ is disjoint from the sutures on $\boundary M$, $K(\ob{Q}) \geq 0$ as noted in the introduction to this section. Since $\ob{Q}$ is a sphere or disc, we also have $-2\chi(\ob{Q}) < 0$. Hence, by the main theorem $(M[\alpha],\gamma)$ is $\nil$--taut. This implies that $M[\alpha]$ is irreducible and that $R_\pm(\gamma)$ does not compress in $M[\alpha]$. Consequently, $c$ does not compress in $M[\alpha]$.
\end{proof}

\begin{remark}
At the cost of adding hypotheses on the embedding of $W$ in $M$, the conditions for being ``licit'' can be significantly weakened. For example, the hypotheses on the curves $c$, $a$, and $b$ of Lemma \ref{Lem: Choosing Sutures} can be substituted for (L2). An examination of the homology argument at the beginning of the proof shows that (L0) can be be replaced with the assumption that $L_\alpha$ and $L_\beta$ are null-homologous in $M$. 
\end{remark}

\section{Rational Tangle Replacement}\label{RTR}
In this section, we show how the Main Theorem combined with Theorem \ref{Thm: Constructing Q} can be used to give new proofs of several theorems concerning rational tangle replacement. Following \cite{EM2}, we define a few relevant terms.

A \defn{tangle} $(B,\tau)$ is a properly embedded pair of arcs $\tau$ in a 3--ball $B$.  Two tangles $(B,\tau)$ and $(B,\tau')$ are \defn{equivalent} if they are homeomorphic as pairs.  They are \defn{equal} if there is a homeomorphism of pairs which is the identity on $\boundary B$.  The \defn{trivial tangle} is the pair $(D^2 \times I, \{.25,.75\} \times I)$. A \defn{rational tangle} is a tangle equivalent to the trivial tangle.  Each rational tangle $(B,r)$ has a disc $D_r \subset B$ separating the strands of $r$ (each of which is isotopic into $\boundary B$).  The disc $D_r$ is called a \defn{trivializing} disc for $(B,r)$.  The \defn{distance} $d(r,s)$ between two rational tangles $(B,r)$ and $(B,s)$ is simply the minimal intersection number $|D_r \cap D_s|$. We will often write $d(D_r,D_s)$ instead of $d(r,s)$. A \defn{prime} tangle $(B,\tau)$ is one without local knots (i.e. every meridional annulus is boundary-parallel) and where no disc in $B$ separates the strands of $\tau$.

Given a knot $L_\beta \subset M$ and a 3--ball $B'$ intersecting $L_\beta$ in two arcs such that $(B',B' \cap L_\beta) = (B',r_\beta)$ is a rational tangle, to replace $(B',r_\beta)$ with a rational tangle $(B',r_\alpha)$ is to do a \defn{rational tangle replacement} on $L_\beta$. Notice that that $\eta(L_\beta) \cup B'$ is a genus 2 handlebody $W$. The knots or links $L_\beta$ and $L_\alpha$ can be obtained by refilling the meridians $\beta$ and $\alpha$ respectively. If $M = S^3$ then $(B,\tau) = (S^3 - \inter{B}',L_\beta - \inter{B}')$ is a tangle. We assume that no component of $L_\beta$ is disjoint from $B$.

Before stating the applications, we state and prove some lemmas which allow the terminology of tangle sums and rational tangle replacement to be converted into the terminology of boring.

\subsection{Boring and Rational Tangle Replacement}\label{RT}

\begin{lemma}\label{Lem: Compressing Tangles}
Let $(B,\tau)$ be a tangle and $N = B - \inter{\eta}(\tau)$. Suppose that $c$ is an essential curve on $\boundary B - \tau$ which separates $\boundary N$. If $\boundary N - c$ is compressible in $N$ then $c$ compresses in $N$.
\end{lemma}
\begin{proof}
Let $d$ be an essential curve in $\boundary N - c$ which bounds a disc $D \subset N$. Since $c$ is separating and $\boundary N$ has genus two, $d$ is a curve in a once-punctured torus. Thus, it is either non-separating or parallel to $c$. In the latter case, we are done, so suppose that $d$ is non-separating. Let $D_+$ and $D_-$ be parallel copies of $D$ so that $d$ is contained in an annulus between $\boundary D_+$ and $\boundary D_-$. Use a loop which intersects $d$ exactly once to band together $D_+$ and $D_-$, forming a disc $D'$. The boundary of $D'$ is an essential separating curve in the once-punctured torus. $\boundary D'$ is, therefore, parallel to $c$. Hence, $c$ compresses in $N$.
\end{proof}

\begin{lemma}\label{Lem: Tangle Compressing}
Suppose that $(B,\tau)$ and $(B',r_\alpha)$ are tangles embedded in $S^3$ with $(B',r_\alpha)$ a rational tangle so that $\boundary B = \boundary B'$ and $\boundary \tau = \boundary r_\alpha$. Suppose that $(B',r_\beta)$ is rational tangle of distance at least one from $(B',r_\alpha)$. Define the sutures $\gamma \cup a$ on $\boundary N$ as before. If
\begin{itemize}
\item $\alpha$ is non-separating in the handlebody $W = B' \cup \eta(\tau)$, or
\item if $(B,\tau)$ is a prime tangle, or 
\item if $(B,\tau)$ is a rational tangle and $\boundary \alpha$ does not bound a trivializing disc for $(B,\tau)$, or
\item if $\boundary \alpha$ does not compress in $(B,\tau)$
\end{itemize}
then $\boundary W - (\gamma \cup a)$ is incompressible in $N$. Consequently, $(N,\gamma \cup a)$ is $\nil$--taut and $(N[a],\gamma)$ is $\ob{\alpha}$--taut.
\end{lemma}

\begin{proof}
If $\alpha$ is non-separating then any compressing disc for $\boundary W - (\gamma \cup \boundary \alpha)$ would have meridional boundary, implying that $S^3$ had a non-separating 2--sphere. Thus, we may suppose that $\alpha$ is separating. If $(B,\tau)$ is prime, there is no disc separating the strands of $\tau$. Similarly, if $(B,\tau)$ is a rational tangle but $a$ does not bound a trivializing disc then $a$ does not compress in $(B,\tau)$. Thus, for the remaining three hypotheses, we may assume that $a$ does not compress in $(B,\tau)$. By Lemma \ref{Lem: Compressing Tangles}, $\boundary N - a$ is incompressible in $N$, as desired. By Lemma \ref{Lem: Choosing Sutures}, $(N,\gamma \cup a)$ is taut and $(N[a],\gamma)$ is $\ob{\alpha}$--taut.
\end{proof}

One pleasant aspect of working with rational tangle replacements is that we can make explicit calculations of $K(\ob{Q})$. Here are two lemmas which we jointly call the \defn{Tangle Calculations}.

\begin{TangleCalcI}[$\beta$ separating]\label{Lem: Tangle Calcs, Sep}
Suppose that $L_\beta$ is a link obtained from $L_\alpha$ by a rational tangle replacement of distance $d$ using $W$. Let $\ob{Q}$ be a suitably embedded surface in the exterior $S^3[\beta]$ of $L_\beta$. Let $\boundary_1 \ob{Q}$ be the components of $\boundary \ob{Q}$ on one component of $\boundary S^3[\beta]$ and $\boundary_2 \ob{Q}$ be the components on the other. Let $n_i$ be the minimum number of times a component of $\boundary_i \ob{Q}$ intersects a meridian of $\boundary S^3[\beta]$.
\begin{itemize}
\item If $L_\alpha$ is a link then
\[
K(\ob{Q}) \geq 2q(d - 1) + d(|\boundary_1 \ob{Q}|n_1 + |\boundary_2 \ob{Q}|n_2).
\]

\item If $L_\alpha$ is a knot then
\[K(\ob{Q}) \geq 2q(d - 1) +  (d-1)(|\boundary_1 \ob{Q}|n_1 + |\boundary_2 \ob{Q}|n_2).
\]
\end{itemize}
\end{TangleCalcI}
\begin{proof}
Since $L_\beta$ is a link, $\beta$ is separating. Thus, $q^* = 0$. Since $a$ and $b$ are contained in $\boundary B' = \boundary B$ every arc of $b - a$ is an meridional arc. Hence, $\nu = 0$. By definition $2d = \Delta$. 

Let $T$ be a component of $\boundary S^3[\beta]$. Without loss of generality, suppose that the components of $\boundary \ob{Q}$ on $T$ are $\boundary_1 \ob{Q}$. Since every arc of $a - b$ is meridional, there exist $d$ meridional arcs on each component of $\boundary S^3[\beta]$. Thus, each component of $\boundary_1 \ob{Q}$ intersects $a$ at least $dn_1$ times. Each component of $\boundary_2 \ob{Q}$ intersects $a$ at least $dn_2$ times. Consequently, $|\boundary_1 \ob{Q} \cap a| \geq |\boundary_1 \ob{Q}|n_1d$. Similarly, $|\boundary_2 \ob{Q} \cap a| \geq |\boundary_2 \ob{Q}|n_2d$. Hence,
\[
\Delta_\boundary \geq d(|\boundary_1 \ob{Q}|n_1 + |\boundary_2 \ob{Q}|n_2).
\]

If $\alpha$ is non-separating, the curves $\gamma$ are also meridian curves of $L_\beta$. Thus, $\gamma$ is intersected $n_i$ times by each component of $\boundary_i \ob{Q}$. Hence, if $L_\alpha$ is a knot,
\[
\nu_\boundary = |\boundary_1 \ob{Q}|n_1 + |\boundary_2 \ob{Q}|n_2.
\]

The result follows.
\end{proof}
\begin{TangleCalcII}[$\beta$ non-separating]
Suppose that $L_\beta$ is a knot obtained from $L_\alpha$ by a rational tangle replacement of distance $d$ using $W$. Let $\ob{Q}$ be a suitably embedded surface in the exterior $S^3[\beta]$ of $L_\beta$. Suppose that each component of $\boundary \ob{Q}$ intersects a meridian of $\boundary S^3[\beta]$ $n$ times.
\begin{itemize}
\item If $L_\alpha$ is a link then
\[
K(\ob{Q}) \geq 2q(d - 1) + 2q^*(2d - 1) + 2d|\boundary \ob{Q}|n.
\]
\item If $L_\beta$ is a knot then
\[
K(\ob{Q}) \geq 2(d-1)(q + 2q^*) + 2(d-1)|\boundary \ob{Q}|n.
\]
\end{itemize}
\end{TangleCalcII}
\begin{proof}
These calculations are similar to the calculations of the previous lemma, so we make only a few remarks. First, since $b^*$ and $\boundary \eta(b)$ cobound a thrice-punctured sphere, every meridional arc of $a-b$ intersects $b^*$ at least twice. Since every arc of $a-b$ is meridional, there are $\Delta$ such arcs. Hence $\Delta^* \geq 4d$. Secondly, if $L_\alpha$ is a knot, then $b^*$ intersects $\gamma$ twice and $b$ intersects $\gamma$ not at all. Thus,
\[
q(\Delta - \nu - 2) + q^*(\Delta^* - \nu^* - 2) \geq q(2d - 2) + q^*(4d - 4).
\]
The given inequality follows.
\end{proof}

\subsection{Discs, Spheres, and Meridional Planar Surfaces}

In \cite{EM2}, Eudave-Mu\~noz states six related theorems. In this section, we give new proofs for three of them. Gordon and Luecke \cite{GLu1} have also given different proofs for some of them. The new proofs will follow from the following generalization. Using completely different sutured manifold theory techniques \cite{T} further extends this theorem.

\begin{theorem}\label{Thm: Meridional Planar}
Suppose that $L_\beta$ is a knot or link obtained by a rational tangle replacement of distance $d \geq 1$ on the split link $L_\alpha$. Suppose that  $\boundary W - \boundary \alpha$ does not compress in $N$. Then $L_\beta$ is not a split link or unknot. Furthermore, if $L_\beta$ has an essential properly embedded meridional planar surface with $m$ boundary components, it contains such a surface $\ob{Q}$ with $|\boundary \ob{Q}| \leq m$ such that either $\ob{Q}$ is disjoint from $\ob{\beta}$ or
\[
|\ob{Q} \cap \ob{\beta}|(d - 1) \leq |\boundary \ob{Q}| - 2
\]
\end{theorem}
\begin{proof}
By Lemma \ref{Lem: Tangle Compressing}, $(N,\gamma \cup a)$ is a taut sutured manifold. Notice that the pair $(S^3,W)$ is licit and that since $L_\alpha$ and $L_\beta$ are related by rational tangle replacement no essential disc in $S^3[\beta]$ is disjoint from $a$. Thus by Theorem \ref{Thm: MSC Unsymmetrized} $L_\beta$ is neither a split link nor an unknot.

Suppose, therefore, that $S^3[\beta]$ contains an essential meridional surface $\ob{R}$ with $m$ boundary components. Use Corollary \ref{Cor: Constructing Q} to obtain the connected planar surface $\ob{Q} \subset S^3[\beta]$ and assume that $\ob{Q}$ is not disjoint from $\ob{\beta}$. That is, assume that $\tild{q} > 0$. Since $\ob{Q}$ is connected and has euler characteristic not lower than our original planar surface, $|\boundary \ob{Q}| \leq m$. The boundary of $\ob{Q}$ is meridional, by construction, since each arc of $a - b$ is meridional. Corollary \ref{Cor: Constructing Q} allows us to conclude that there is no compressing disc, $a$--boundary compressing disc, or $a$--torsion $2g$--gon for $Q$. Also, $S^3[\alpha]$ is reducible and $H_1(S^3[\alpha])$ is torsion-free.

The Main Theorem concludes, therefore, that $K(\ob{Q}) \leq -2\chi(\ob{Q})$. Since $\boundary \ob{Q}$ is disjoint from $a \cup \gamma$ and since $L_\alpha$ is a link the tangle calculations tell us that:
\[
2q(d-1) + 2q^*(2d-1) \leq -2\chi(\ob{Q}).
\]

Since $4q^*(d-1) \leq 2q^*(2d-1)$, we may conclude that $2(q + 2q^*)(d-1) \leq -2\chi(\ob{Q})$. $\ob{Q}$ is a planar surface with $|\boundary \ob{Q}|$ boundary components, implying that $-2\chi(\ob{Q}) = 2|\boundary \ob{Q}| - 4$. Plugging into our inequality and dividing by two, we obtain
\[
(q + 2q^*)(d-1) \leq |\boundary \ob{Q}| - 2.
\]

A slight isotopy pushing the discs in $\ob{Q}$ with boundary parallel to $b^*$ converts each such disc into two discs each with boundary parallel to $b$. Hence, after the isotopy $|\ob{Q} \cap \ob{\beta}| = q + 2q^*$. Consequently,
\[
|\ob{Q} \cap \ob{\beta}|(d-1) \leq |\boundary \ob{Q}| - 2
\]
as desired.
\end{proof}

As corollaries, we have the following classical results. 

\begin{theorem*}[Eudave-Mu\~noz \cite{EM2}]
If $(B,\tau)$ is prime, if $L_\alpha$ is a split link, and if $L_\beta$ is composite then $d(\alpha,\beta) \leq 1$.
\end{theorem*}
\begin{proof}
Suppose that $d \geq 1$. Since $(B,\tau)$ is prime and $\alpha$ is separating, Lemma \ref{Lem: Tangle Compressing} shows that $\boundary W - a$ is incompressible in $N$. Since $L_\beta$ contains an essential meridional annulus, we may apply Theorem \ref{Thm: Meridional Planar} with $m = 2$. Since there are no meridional discs, $\ob{Q}$ is also a meridional annulus. Since $(B,\tau)$ is prime, $\ob{Q}$ is not disjoint from $\ob{\beta}$. The inequality from the theorem, shows that $d = 1$.
\end{proof}

\begin{theorem*}[Eudave-Mu\~noz \cite{EM2}]
If $(B,\tau)$ is any tangle and if $L_\alpha$ and $L_\beta$ are split links, then $r_\alpha = r_\beta$.
\end{theorem*}
\begin{proof}
It suffices to show that $\alpha$ and $\beta$ are disjoint. Suppose not, so that $d \geq 1$. If $\boundary W - a$ is incompressible in $N$ then by Theorem \ref{Thm: Meridional Planar} $L_\beta$ is not a split link. Thus, $\boundary W - a$ compresses in $N$. By reversing the roles of $\alpha$ and $\beta$ we can also conclude that $\boundary W - b$ compresses in $N$. Since both $\alpha$ and $\beta$ are separating, Lemma \ref{Lem: Compressing Tangles} shows that both $a$ and $b$ compress in $N$.

There is, therefore, a disc $D_a$ in $B$ with boundary $a$ separating the strings of $\tau$. Similarly, there is a disc $D_b$ in $B$ with boundary $b = \boundary \beta$ separating the strings of $\tau$. An easy innermost disc/outermost arc argument shows that $D_a$ and $D_b$ are isotopic. In particular, $a$ and $b$ are isotopic in $\boundary B - \tau$ which implies that $r_\alpha = r_\beta$.

Thus we may assume, without loss of generality, that $\boundary W - \boundary \alpha$ is not compressible in $N$. Let $\ob{R}$ be an essential sphere in $S^3[\beta]$ and apply Corollary \ref{Cor: Constructing Q} to obtain an essential sphere or disc $\ob{Q}$. Since $a-b$ consists of meridional arcs, $\ob{Q}$ is not disjoint from $\eta(a)$. If $\ob{Q}$ were a disc disjoint from $\ob{\beta}$, there would be no $a$--boundary compressing disc for $\ob{Q}$. If $\ob{Q}$ is a sphere, $\tild{q} > 0$. Thus, we may apply the main theorem to conclude that $S^3[\alpha]$ is irreducible or that $\alpha$ and $\beta$ are disjoint. If the latter is true, $r_\alpha = r_\beta$.
\end{proof}

\begin{theorem*}[Scharlemann \cite{S1}]
If $(B,\tau)$ is any tangle and $L_\beta$ is a trivial knot and $L_\alpha$ a split link then $(B,\tau)$ is a rational tangle and $d \leq 1$.
\end{theorem*}

\begin{proof}
Suppose $d \geq 1$. If $\boundary W - a$ were incompressible in $N$ then, by Theorem \ref{Thm: Meridional Planar}, $L_\beta$ would not be the unknot. Hence $\boundary W - a$ is compressible in $N$. Since $\alpha$ is separating, Lemma \ref{Lem: Compressing Tangles} shows that $a$ compresses in $N$. Since $L_\beta$ is the unknot, $\tau$ has no local knots. Thus, $(B,\tau)$ is a rational tangle with trivializing disc having boundary $a$.  

It remains to prove that $d = 1$. Since $L_\beta$ is the unknot, a double-branched cover of $S^3$ with branch set $L_\beta$ is $S^3$. The preimage $\tild{B}$ of $B$ is an unknotted solid torus. There is a correspondence between rational tangle replacement and Dehn-surgery in the double-branched cover. Replacing $(B',r_\beta)$ with $(B',r_\alpha)$ converts the double-branched cover to a lens space, $S^3$ or $S^1 \times S^2$. In the double branched cover, the Dehn surgery is achieved by making a curve in $\boundary \tild{B}$ which intersects a meridian of $\tild{B}$ $d$ times bound a disc in the complementary solid torus. Since $L_\alpha$ is a split link, the double branched cover of $S^3$ over $L_\alpha$ is reducible. Thus, it must be $S^1 \times S^2$ and $d$ must be one, as desired.
\end{proof}

\begin{remark}
In the proof of the previous theorem, note that even without proving $d \leq 1$, we have provided a new proof of Scharlemann's band sum theorem \cite{S1}: If $K = K_1 \#_b K_2$ is the unknot then the band sum is the connected sum of unknots. To see this note that $W$ is $\eta(K_1 \cup K_2 \cup b)$ where $b$ is the band. The tangle $(B,\tau)$ is  $(S^3 - \inter{\eta}(b), (K_1 \cup K_2) - \inter{\eta}(b))$. Since $\boundary \beta$ is a loop which encircles the band, $\boundary \beta$ only bounds a disc in $(B,\tau)$ when the band sum is a connected sum and $K_1$ and $K_2$ are unknots.
\end{remark}

\cite{T} gives other significant applications of sutured manifold theory to problems involving rational tangle replacement.

%%%%%%%%%%% Surface Intersections %%%%%%%%%%%%%%%%%%%%
\section{Intersections of $\nil$--taut Surfaces} \label{Intersections}

The main theorem is particularly useful for studying a homology class in $H_2(N[a],\boundary N[a])$ which is not represented by a surface disjoint from $\ob{\alpha}$. The propositions of this section consist of observations which can dramatically simplify the combinatorics of such a situation. Let $N$ be a compact, orientable 3--manifold with $F \subset \boundary M$ a genus 2 boundary component. Let $a,b \subset F$ be essential curves which cannot be isotoped to be disjoint and suppose that $(N[a],\gamma)$ is $\ob{\alpha}$--taut, as in Section \ref{Sutures}.

\subsection{Intersection Graphs}
\begin{proposition}\label{Prop: Nil Taut}
Let $(N[a],\gamma)$ and $b$ be as above and suppose that $z \in H_2(N[a],\boundary N[a])$ is a non-trivial homology class. Suppose that $N[a]$ does not contain an essential disc disjoint from $\ob{\alpha}$. Then $z$ is represented by an embedded conditioned $\ob{\alpha}$--taut surface $\ob{P}$. Furthermore, for any such $\ob{P}$ , either $\ob{P}$ is disjoint from $\ob{\alpha}$ or $P = \ob{P} \cap N$ has no compressing discs, $b$--boundary compressing discs or $b$--torsion $2g$--gons. 
\end{proposition}

\begin{proof} 
Let $\ob{P}$ be a conditioned $\ob{\alpha}$--taut surface. (Such a surface is guaranteed to exist by Theorem 2.6 of \cite{S3}.) Suppose that $\ob{P}$ is not disjoint from $\ob{\alpha}$. Recall from the definition of ``$\ob{\alpha}$--taut'' that $\ob{\alpha}$ intersects $\ob{P}$ always with the same sign. Because $\ob{P}$ is $\ob{\alpha}$--taut, $P$ is incompressible. Suppose that $D$ is a $b$--torsion $2g$--gon for $P$. If $g = 1$, $D$ is a $b$--boundary compressing disc for $P$. Let $\epsilon_i$ be the arcs $\boundary D \cap F$. Let $R$ be the rectangle containing the $\epsilon_i$ from the definition of $b$--torsion $2g$--gon. Suppose that the ends of $R$ are on components of $\boundary P - \boundary \ob{P}$. The endpoints of the $\epsilon_i$ have signs arising from the intersection of $\boundary D$ with $\boundary P$. Since $\ob{\alpha}$ always intersects $\ob{P}$ with the same sign an arc $\epsilon_i$ has the same sign of intersection at both its head and tail. Since the arcs are all parallel, all heads and tails of all the $\epsilon_i$ have the same sign of intersection. However, an arc of $\boundary D \cap P$ must have opposite signs of intersection, arising as it does from the intersection of two surfaces. This implies that the head of some $\epsilon_i$ has a sign different from the tail of some $\epsilon_i$, a contradiction. Hence, at least one end of $R$ must lie on a component of $\boundary \ob{P}$.

If one end of $R$ is on $\boundary P - \boundary \ob{P}$ denote that component by $a_1$ and call the disc which it bounds in $\ob{P}$, $\alpha_1$. If both ends of $R$ are on $\boundary \ob{P}$, let $\alpha_1 = \nil$. Attach $R$ to $\ob{P} - \alpha_1$ creating a surface $\tild{P}$. The disc $D$ is contained in $N$ and, therefore, had interior disjoint from $\ob{\alpha}$. Compress $\tild{P}$ using $D$ and continue to call the result $\tild{P}$.

An easy calculation shows that if $\alpha_1 \neq \nil$, then $\chi(\tild{P}) = \chi(\ob{P})$ but $|\ob{\alpha} \cap \tild{P}| = |\ob{\alpha} \cap \ob{P}| - 1$. Similarly, if $\alpha_1 = \nil$, then $-\chi(\tild{P}) = -\chi(\ob{P}) - 1$ and $|\ob{\alpha} \cap \tild{P}| = |\ob{\alpha} \cap \ob{P}|$. If $\chi_{\ob{\alpha}}(\ob{P}) \neq |\ob{\alpha} \cap \ob{P}| - \chi(\ob{P})$ then a component of $\ob{P}$ is a disc disjoint from $\ob{\alpha}$ or a sphere intersected by $\ob{\alpha}$ once. Either of these contradict our hypotheses on $N[a]$. Suppose, therefore, that $\chi_{\ob{\alpha}}(\ob{P}) = |\ob{\alpha} \cap \ob{P}| - \chi(\ob{P})$. 

Similarly, $\chi_{\ob{\alpha}}(\tild{P}) = |\ob{\alpha} \cap \tild{P}| - \chi(\tild{P})$. Hence, $\chi_{\ob{\alpha}}(\tild{P}) = \chi_{\ob{\alpha}}(\ob{P}) - 1$. Since $\ob{\alpha}$ always intersects $\tild{P}$ with the same sign, $\ob{P}$ is not $\ob{\alpha}$--taut, a contradiction. Hence, there are no $b$--torsion $2g$--gons for $P$.
\end{proof}

\begin{remark}
As Scharlemann notes in \cite{S5}, when $a$ and $b$ are non-separating it can be difficult to use combinatorial methods to analyze the intersection of surfaces in $N[a]$ and $N[b]$. The primary reason for this is the need to work with $a^*$ and $b^*$ boundary components on the surfaces. The previous proposition shows that when the surfaces in question are $\ob{\alpha}$--taut and $\ob{\beta}$--taut and not disjoint from $\ob{\alpha}$ and $\ob{\beta}$, respectively, there is no need to consider $a^*$ and $b^*$ curves.
\end{remark} 

The remainder of this section develops notation for studying the intersection graphs of such surfaces. Let $\ob{P} \subset N[a]$ be an $\ob{\alpha}$--taut surface and let $\ob{Q} \subset N[b]$ be a $\ob{\beta}$--taut surface. Suppose that $\ob{P}$ and $\ob{Q}$ are not disjoint from $\ob{\alpha}$ and $\ob{\beta}$ respectively. Suppose also that there is no $b$--torsion $2g$--gon for $P = \ob{P} \cap N$ and no $a$--torson $2g$--gon for $Q = \ob{Q} \cap N$. It is clear that $P$ and $Q$ are incompressible.

In Section \ref{2--handle}, we defined intersection graphs between $\ob{Q}$ and a disc $D$. We now define, in a similar fashion, intersection graphs between $\ob{P}$ and $\ob{Q}$. Orient $P$ (respectively, $Q$) so that all boundary components of $\boundary P - \boundary\ob{P}$ ($\boundary Q - \boundary\ob{Q}$, respectively) are parallel on $\eta(\ob{\alpha})$ ($\eta(\ob{\beta}$), respectively).  The intersection of $P$ and $Q$ forms graphs $\Lambda_\alpha$ and $\Lambda_\beta$ on $\ob{P}$ and $\ob{Q}$. A component of $\boundary P - \boundary\ob{P}$ or $\boundary Q - \boundary \ob{Q}$ is called an \defn{interior boundary component}.  The vertex of $\Lambda_\alpha$ or $\Lambda_\beta$ to which it corresponds is called an \defn{interior vertex}.  

Label the components of $\boundary Q \cap \eta(a)$ as $1, \hdots, \mu_Q$ and the components of $\boundary P \cap \eta(b)$ as $1, \hdots, \mu_P$.  The labels should be in order around $\eta(a)$ and $\eta(b)$.  An endpoint of an edge on an interior vertex of $\Lambda_\alpha$ corresponds to an arc of $\boundary Q \cap \boundary \eta(\alpha)$.  Give the endpoint of the edge the label associated to that arc.  Similarly, label all endpoints of edges on interior vertics of $\Lambda_\beta$.  A \defn{Scharlemann cycle} is a type of cycle which bounds a disc in $\ob{P}$ ($\ob{Q}$, respectively).  The interior of the disc must be disjoint from $\Lambda_\alpha$ ($\Lambda_\beta$) and all of the vertices of the cycle must be interior vertices.  Furthermore, the cycle can be oriented so that the tail end of each edge has the same label.  This is the same notion of Scharlemann cycle as in Section \ref{2--handle}, but adapted to the, possibly non-planar, surfaces $\ob{P}$ and $\ob{Q}$.

\begin{lemma}\label{Lem: No Cycles}
There is no Scharlemann cycle in $\Lambda_\alpha$ or $\Lambda_\beta$. 
\end{lemma}

\begin{proof}
Were there a trivial loop at an interior vertex or a Scharlemann cycle in $\Lambda_\alpha$ or $\Lambda_\beta$, the interior would be an $a$ or $b$--torsion $2g$--gon, contradicting Proposition \ref{Prop: Nil Taut}.
\end{proof}

Although we will not use it here, the next lemma may, in the future, be a useful observation. 

\begin{lemma}\label{no loops}
If $\ob{P}$ is a disc, then every loop in $\Lambda_\alpha$ is based at $\boundary \ob{P}$.
\end{lemma}
\begin{proof}
Suppose that $\ob{P}$ is a disc and that there is a loop based at an interior vertex of $\Lambda_\alpha$. A component $X$ of the complement of the loop in $\ob{P}$ does not contain $\boundary \ob{P}$. The loop is an $x$--cycle and Lemma \ref{Scharlemann exists} then guarantees the existence of a Scharlemann cycle in $X$, contrary to Lemma \ref{Lem: No Cycles}. 
\end{proof}

%%%%%%%%%%%%%%%% Anannular Exterior %%%%%%%%%%%%%%%%%%%%%%%%%%%%%

\subsection{When the exterior of $W$ is anannular.}

We conclude this section with an application to refilling meridians of a genus 2 handlebody whose exterior is irreducible, boundary-irreducible, and anannular. It is based on the ideas in \cite{SW}. Suppose that $M$ is the exterior of a link in $S^3$. Suppose that $W \subset M$ is a genus 2 handlebody embedded in $M$. Let $N = M - \inter{W}$.

\begin{theorem}\label{Simple Ineq}
Suppose that $N$ is irreducible, boundary-irreducible and anannular. Suppose that $\alpha$ and $\beta$ are non-separating meridians of $W$ such that $\Delta > 0$. Suppose that neither $M[\alpha]$ nor $M[\beta]$ contain an essential disc or sphere. Suppose also that in $H_2(M[\alpha],\boundary M)$ there is a homology class $z_a$ which cannot be represented by a surface disjoint from $\ob{\alpha}$ and that in $H_2(M[\beta],\boundary M)$ there is a homology class $z_b$ which cannot be represented by a surface disjoint from $\ob{\beta}$. Then there is a $\nil$--taut surface $\ob{P} \subset M[\alpha]$ representing $z_a$ intersecting $\ob{\alpha}$ $p$ times and an $\nil$--taut surface $\ob{Q} \subset M[\beta]$ representing $z_b$ intersecting $\ob{\beta}$ $q$ times such that one of the following occurs:
\begin{enumerate}
\item $-2\chi(\ob{P}) \geq p(\mc{M}_b(a) - 2)$
\item $-2\chi(\ob{Q}) \geq q(\mc{M}_a(b) - 2)$
\item All of the following occur:
\begin{itemize}
\item $\ob{Q}$ is $\ob{\beta}$--taut
\item $\ob{P}$ is $\ob{\alpha}$--taut.
\item $pq\Delta \leq 18(p - \chi(\ob{P}))(q - \chi(\ob{Q}))$
\item $\Delta < \frac{9}{2}\mc{M}_a(b)\mc{M}_b(a)$
\end{itemize}
\end{enumerate}
\end{theorem}
\begin{proof}
Notice that the right hand side of the inequalities in (1) and (2) are $K(\ob{P})$ and $K(\ob{Q})$ respectively. Choose a taut representative in $M[\beta]$ for $z_b$ and apply Theorem \ref{Thm: Constructing Q}, obtaining $\ob{Q}$. Since negative euler characteristic is not increased and $M[\beta]$ does not contain an essential disc or sphere, $\ob{Q}$ is also taut. If (1) holds, we are done, so assume that $-2\chi(\ob{Q}) < K(\ob{Q})$. Recall that $\ob{Q}$ is not disjoint from $\ob{\alpha}$. Apply the main theorem to obtain a surface $\ob{P} \subset M[\alpha]$ representing $z_a$. (The surface $\ob{P}$ is the surface $S$ in the statement of that theorem.) $\ob{P}$ is both $\ob{\alpha}$--taut and $\nil$--taut. If (2) holds, we are done, so assume $-2\chi(\ob{P}) < K(\ob{P})$. Applying the main theorem again, with $\alpha$ and $\beta$ reversed, we find a $\ob{\beta}$--taut and $\nil$--taut surface in $M[\beta]$ representing $z_b$. We may call this surface $\ob{Q}$, forgetting the previous one. Consider the the graphs formed by the intersection of $P$ and $Q$; let $\Lambda_\alpha$ be the graph on $\ob{P}$ and $\Lambda_\beta$ the graph on $\ob{Q}$. Lemma \ref{Lem: No Cycles} assures us that there is no trivial loop based at an interior vertex of either graph.

\begin{lemma}\label{Lem: SW Ineq}
\[
pq\Delta \leq 18(p - \chi(\ob{P}))(q - \chi(\ob{Q}))
\]
\end{lemma}
\begin{proof}[Proof of Lemma \ref{Lem: SW Ineq}]
By \cite[Lemma 2.1]{SW}, if two edges of $P \cap Q$ are parallel in both $\Lambda_\alpha$ and $\Lambda_\beta$, there is an essential annulus in $N$, contrary to our assumption that $N$ is anannular.  The proof proceeds as in \cite{SW}.  

Each interior boundary component of $P$ intersects $\boundary Q$, $q\Delta$ times. Thus $|\boundary Q \cap \boundary P| \geq pq\Delta$. Therefore, $\Lambda_\alpha$ and $\Lambda_\beta$ each have at least $pq\Delta/2$ edges.

\textbf{Claim: } $\Lambda_\alpha$ has at least $\frac{pq\Delta}{6(p - \chi(\ob{P}))}$ mutually parallel edges.

This claim is similar to work in \cite{GLi}. Let $\Lambda'$ be the graph obtained by combining each set of parallel edges of $\Lambda_\alpha$ into a single edge. Since $\Lambda'$ has no loops at interior vertices and no parallel edges, by applying the formula for the euler characteristic of a closed surface we obtain:
\[
\begin{matrix}
\chi(\ob{P}) + |\boundary \ob{P}| &=&  V - E + F\\
&\leq& p + |\boundary \ob{P}| - E + (2/3)E\\
&= & p + |\boundary \ob{P}| - (1/3) E\\
\end{matrix}
\]
$V$, $E$, and $F$ represent the number of vertices, edges, and faces of $\Lambda'$. Thus, $E \leq 3(p - \chi(\ob{P}))$. Let $n$ be the largest number of mutually parallel edges in $\Lambda_\alpha$. Then, since there are at least $pq\Delta/2$ edges in $\Lambda_\alpha$, we have
\[
pq\Delta/(2n) \leq E \leq 3(p - \chi(\ob{P})).
\]
The claim follows.

A similar argument shows that if a graph in $\ob{Q}$ has more than $3(q - \chi(\ob{Q}))$ edges than two of them are parallel.  Hence, since there are no mutually parallel edges in $\Lambda_\alpha$ and $\Lambda_\beta$ we must have:

$$\frac{pq\Delta}{6(p - \chi(\ob{P}))} \leq 3(q - \chi(\ob{Q}))$$ whence the lemma and the first inequality of conclusion 3 of the theorem follow.
\end{proof}

We now proceed with the proof of the theorem. Since we are assuming that neither (1) nor (2) hold, we have
\[
\begin{matrix}
-\chi(\ob{P}) & < & K(\ob{P})/2 &=& p(\mc{M}_b(a) - 2)/2\\
&&&&\\
-\chi(\ob{Q}) & < & K(\ob{Q})/2 &=& q(\mc{M}_a(b) - 2)/2
\end{matrix}
\]

Plugging into the inequality from the lemma, we obtain
\[
pq\Delta < 18pq\Big(1 + \frac{\mc{M}_b(a) - 2}{2}\Big)\Big(1 + \frac{\mc{M}_a(b) - 2}{2}\Big).
\]
Since neither $p$ nor $q$ is zero, we divide and simplify to obtain:
\[
\Delta < 9\mc{M}_b(a)\mc{M}_a(b)/2.
\]
\end{proof}

\begin{remark}
The point of the previous theorem is that, under the specified conditions, either we obtain a bound on the euler characteristic of surfaces representing the homology classes $z_a$ or $z_b$ or we obtain a restriction on the number of non-meridional arcs of $a - b$ and $b-a$. For example, suppose that discs $\alpha$ and $\beta$ are chosen so that $z_a$ is represented by a once-punctured torus, and so that $\mc{M}_b(a) = \mc{M}_a(b) = 6$. Then $-2\chi(\ob{P}) = 2 < 4p = K(\ob{P})$. Then if $z_b$ is also represented by a once punctured torus, we have $\Delta < 162$. Since $\Delta$ is even, this implies $\Delta \leq 160$.
\end{remark}

 \end{document}